\documentclass[11pt, oneside]{amsart}
 \usepackage[text={5.58in,8.5in},centering,letterpaper,dvips]{geometry}
\usepackage{graphicx}
\usepackage{amsfonts}
\usepackage[dvipsnames]{xcolor}
\usepackage{subcaption}
\usepackage{epsf}
\usepackage{amssymb}
\usepackage{amsmath}
\usepackage{amscd}
\usepackage{pdfpages}
\usepackage{fancyhdr}
\usepackage{setspace}
\usepackage{soul} 
\usepackage[all]{xy}
\usepackage{verbatim}
\usepackage{enumerate}
\usepackage[colorlinks=true, urlcolor=NavyBlue, linkcolor=NavyBlue, citecolor=NavyBlue]{hyperref}

\theoremstyle{theorem}
\newtheorem{theorem}{Theorem}[section]
\newtheorem{proposition}[theorem]{Proposition}
\newtheorem{lemma}[theorem]{Lemma}
\newtheorem{question}[theorem]{Question}
\newtheorem{corollary}[theorem]{Corollary}
\newtheorem{conjecture}[theorem]{Conjecture}

\makeatletter
\newtheorem*{rep@theorem}{\rep@title}
\newcommand{\newreptheorem}[2]{%
\newenvironment{rep#1}[1]{%
 \def\rep@title{#2 \ref{##1}}%
 \begin{rep@theorem}}%
 {\end{rep@theorem}}}
\makeatother

\newreptheorem{theorem}{Theorem}
\newreptheorem{lemma}{Lemma}
\newreptheorem{question}{Question}
\newreptheorem{corollary}{Corollary}
\newreptheorem{proposition}{Proposition}

\theoremstyle{definition}

\newtheorem{remark}[theorem]{Remark}

\newcommand{\Rr}{\mathfrak R}

\newcommand{\A}{\alpha}

\newcommand{\pd}{\partial}

\newcommand{\K}{\kappa}
\newcommand{\D}{\mathcal D}
\newcommand{\cl}{\text{cl}}

\makeatletter
\def\@seccntformat#1{%
  \protect\textup{\protect\@secnumfont
    \ifnum\pdfstrcmp{subsection}{#1}=0 \bfseries\fi
    \csname the#1\endcsname
    \protect\@secnumpunct
  }%
}  
\makeatother

\begin{document}

\rhead{\thepage}
\lhead{\author}
\thispagestyle{empty}


\raggedbottom
\pagenumbering{arabic}
\setcounter{section}{0}


\title{Symmetric ribbon numbers of low-complexity knots}

\author{Sajid Raihan Akash}
\address{Department of Physics and Astronomy, University of Nebraska--Lincoln, Lincoln, NE 68588}
\email{sakash2@huskers.unl.edu}

\author{Eric Corrado}
\address{Department of Mathematics, University of Nebraska--Lincoln, Lincoln, NE 68588}
\email{ericanthonycorrado@gmail.com}

\author{Bishop Placke}
\address{Department of Mathematics, University of Nebraska--Lincoln, Lincoln, NE 68588}
\email{bplacke4@huskers.unl.edu}

\author{Sam Sanketh}
\address{Department of Mathematics, University of Nebraska--Lincoln, Lincoln, NE 68588}
\email{ssanketh2@huskers.unl.edu}

\author{Nick Starns}
\address{Department of Mathematics, University of Nebraska--Lincoln, Lincoln, NE 68588}
\email{nickstarns.cs@gmail.com}

\author{Anok Timothy}
\address{Department of Mathematics, University of Nebraska--Lincoln, Lincoln, NE 68588}
\email{atimothy3@huskers.unl.edu}

\author{Alexander Zupan}
\address{Department of Mathematics, University of Nebraska--Lincoln, Lincoln, NE 68588}
\email{zupan@unl.edu}
\urladdr{https://math.unl.edu/azupan2}


\begin{abstract}
Every knot $K \subset S^3$ that admits a symmetric union presentation bounds an immersed ribbon disk in $S^3$, while the converse is an open problem due to Christoph Lamm.  The symmetric ribbon number $r_s(K)$ of $K$ is the minimum number of ribbon singularities in any symmetric ribbon disk bounded by $K$.  In this paper, we undertake a systematic investigation of symmetric ribbon numbers of knots with at most 12 crossings.  Along the way, we exhibit novel lower bounds for $r_s(K)$ arising from knot determinants, Alexander polynomials, Jones polynomials, and Kauffman polynomials.
\end{abstract}

\maketitle

\section{Introduction}\label{sec:intro}

In~\cite{lamm1}, Lamm introduced the notion of a symmetric union presentation for a knot $K$, constructed by beginning with a symmetric diagram $D \# -\!\!D$ and inserting additional crossings along the axis of symmetry, as shown at left in Figure~\ref{fig:symm}.

\begin{figure}[h!]
    \centering
    \includegraphics[width=0.9\textwidth]{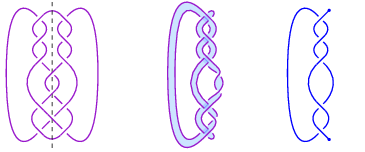}
            \put (-30,71) {$-2$}
            \put (-34,32) {$1$}
                 \caption{At left, a symmetric union presentation for the knot $12n_{313}$, appearing in~\cite{lamm2}.  At center, the corresponding symmetric ribbon disk, whose diagram can be obtained by ``folding" the left picture in half across the axis of symmetry.  At right, the symmetric ribbon disk is encoded as a \emph{labeled knotoid diagram}.}
    \label{fig:symm}
\end{figure}

Lamm observed that if $K$ admits a symmetric union presentation, then $K$ is a ribbon knot, posing the converse as a still-open question:

\begin{question}
Does every ribbon knot admit a symmetric union presentation?
\end{question}

Aceto reframed the idea of a symmetric union presentation by observing that every ribbon surface is isotopic to a ``flattened" surface built from the pieces shown in Figure~\ref{fig:modular}, proving that a ribbon disk $\D$ arises from Lamm's symmetric union construction if and only if $\D$ can be built without crossings or junctions~\cite{aceto}.  We call such a disk a \emph{symmetric ribbon disk}, and a knot $K$ is a \emph{symmetric ribbon knot} if $K$ bounds a symmetric ribbon disk.  An example of a symmetric ribbon disk appears at center in Figure~\ref{fig:symm}.  Lamm showed that each of the 21 prime ribbon knots $K$ with $c(K) \leq 10$ is symmetric ribbon knot~\cite{lamm1,lamm4} and then proved the same for all but 15 of the 137 prime ribbon knots with crossing number 11 or 12~\cite{lamm2}.

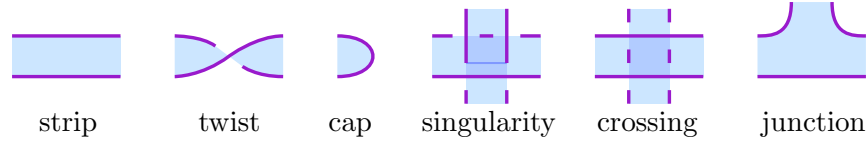
\begin{figure}[h!]
    \centering
    \includegraphics[width=0.8\textwidth]{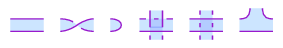}
            \put (-324,0) {strip}
            \put (-264,0) {twist}
            \put (-215,0) {cap}
            \put (-180,0) {singularity}
            \put (-114,0) {crossing}
            \put (-52,0) {junction}            
     \caption{The fundamental pieces used to construct a modular ribbon disk.}
    \label{fig:modular}
\end{figure}

For a ribbon knot $K$, the ribbon number $r(K)$ is defined to be the minimum number of ribbon singularities contained in any ribbon disk $\D$ bounded by $K$.  Ribbon number was introduced in~\cite{miz} and tabulated for prime knots up to 11 crossings in~\cite{FMZ} and for prime knots with 12 crossings in~\cite{polymath}.  Aceto defined the analogue for a symmetric ribbon knot $K$, the \emph{symmetric ribbon number} $r_s(K)$, to be the minimum number of ribbon singularities in any symmetric ribbon disk bounded by $K$.

In this article, we explore various lower bounds for $r_s(K)$ arising from the determinant, Alexander polynomial, Jones polynomial, and Kauffman polynomial of a knot $K$.  We apply our tool to low-complexity knots, proving

\begin{theorem}\label{thm:main}
The symmetric ribbon numbers for all but 3 of the 44 known prime symmetric ribbon knots with 11 or fewer crossings appear in Table~\ref{table:10}, the symmetric ribbon numbers for all but 4 of the 53 known prime nonalternating symmetric ribbon knots with crossing number 12 appear in Table~\ref{table:12n}, and the symmetric ribbon numbers for 27 of the 46 known prime alternating symmetric ribbon knots with crossing number 12 appear in Table~\ref{table:12a}.
\end{theorem}


While it is clear from the definitions that $r(K) \leq r_s(K)$, our work yields the first examples of symmetric ribbon knots $K$ such that $r(K) < r_s(K)$.  Of the 105 prime symmetric ribbon knots up to 12 crossings such that both $r(K)$ and $r_s(K)$ are known, 56 satisfy $r(K) = r_s(K)$, 45 satisfy $r_s(K) - r(K) = 1$, and four, the knots $12a_{427}$, $12a_{435}$, $12a_{464}$, and $12a_{631}$, satisfy $r_s(K) - r(K) = 2$.  Our data suggest several conjectures.

\begin{conjecture}
For every positive integer $n$, there exists a symmetric ribbon knot $K$ such that $r_s(K) - r(K) = n$.
\end{conjecture}

All of the knots in our data for which $r_s(K) - r(K) > 0$ have $r(K) \geq 3$.  Thus, we also expect

\begin{conjecture}
If $K$ is a symmetric ribbon knot such that $r(K) = 2$, then $r_s(K) = 2$ as well.
\end{conjecture}

Note that Aceto proved for every $n$, there is a ribbon knot $K_n$ such that $r(K_n) = 2$ while $r_s(K_n) > n$~\cite{aceto}.  However, it is unknown whether the knots in Aceto's family are symmetric ribbon knots, and so this inequality admits the possibility that $r_s(K_n) = \infty$ (the convention for a ribbon knot that is not symmetric, should such a knot exist).

To determine the data referenced in Theorem~\ref{thm:main} we encode symmetric ribbon disks as \emph{labeled knotoid diagrams} (defined below, example shown at right in Figure~\ref{fig:symm}) and employ a number of different tactics.  The first of our tools is the set of Alexander polynomials,
\[ \Rr^s_n = \{\Delta_K(t) : r_s(K) = n \text{ and $K$ is prime}\}.\]
A similar set, $\Rr_n$, involving $r(K)$ was introduced in~\cite{FMZ}, in which the authors proved that $\Rr_n$ is finite for all $n$ and determined the sets $\Rr_2$ and $\Rr_3$.  Since $\Rr^s_n \subset \Rr_n$, we have that $\Rr^s_n$ is also finite.  We enumerate $\Rr^s_2$ and $\Rr^s_3$ in Lemmas~\ref{lem:r2} and~\ref{lem:r3} and prove

\begin{theorem}\label{thm:r4}
The set $\Rr^s_4$ consists of the 27 polynomials given in Table~\ref{table:r4}.
\end{theorem}

If $r_s(K) = 5$, the set $\Rr^s_5$ is too large to compute via our methods.  However, leveraging a Theorem of Kidwell about the $z$-degree of the Kauffman polynomial~\cite{kidwell}, we obtain a restriction on these knots involving their determinants.

\begin{theorem}\label{thm:det}
Suppose $K$ is a prime symmetric ribbon knot.  If $r_s(K) = 5$, then $\det(K) \leq 169$.
\end{theorem}

Finally, we obtain additional restrictions via Jones polynomials.  We show

\begin{theorem}\label{thm:jones}
Suppose $K$ is a symmetric ribbon knot such that $r_s(K) = 2$, or $r_s(K) = 3$ and $\det(K) = 25$.  Then there exists an integer $n$ such that
\[ V_K(t) = (-t)^n \cdot f(t) + 1,\]
where
\[ f(t) = -t^{-3}+t^{-2}-t^{-1}+2-t + t^2 - t^3 \text{ if } r_s(K) = 2\] or \[ f(t) = t^{-4}-2t^{-3}+3t^{-2}-4t^{-1}+4-4t+3t^2-2t^3+t^4 \text{ if  } r_s(K) = 3 \text{ and } \det(K) = 25.\]
\end{theorem}

Theorem~\ref{thm:jones} is a combination of Propositions~\ref{prop:jones2} and~\ref{prop:jones3}, and it has both obstructive and constructive applications.  For the knots
\[ K \in \{11n_{39}, 12n_{256}, 12n_{257},12n_{394},12n_{870}\},\]
the Alexander polynomial $\Delta_K(t)$ is in $\Rr^s_3$ but not in $\Rr^s_2$, so that the Alexander polynomial bound shows that $r_s(K) \geq 3$.  However, $\det(K) = 25$ but $V_K(t)$ is not of the form of Theorem~\ref{thm:jones} (as shown in Corollary~\ref{cor:1139}), and so we get the stronger bound $r_s(K) \geq 4$.  On the other hand, for the knots $K = 11n_{37}$ or $12n_{414}$, symmetric union presentations from~\cite{lamm2} indicate that $r_s(K) \leq 4$.  Yet $\det(K) = 25$ and $V_K(t)$ satisfies the conclusion of Proposition~\ref{prop:jones3}, which contains a stronger statement and indicates precisely where to look to find a symmetric ribbon disk for $K$ with three ribbon intersections, should it exist.  Via experiment, we found the new symmetric union disks depicted in Figure~\ref{fig:symm3}, verifying that $r_s(K) = 3$ for these two knots.  More details are included in Remark~\ref{rmk:search}.

\begin{figure}[h!]
    \centering
    \includegraphics[width=0.7\textwidth]{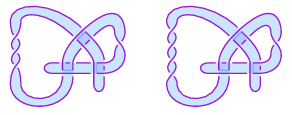}
     \caption{Symmetric ribbon disks for $11n_{37}$ (left) and $12n_{414}$ (right).}
    \label{fig:symm3}
\end{figure}

Our work relied heavily on computations performed using SnapPy~\cite{snappy} within SageMath~\cite{sage} and the knot data available through KnotInfo~\cite{knotinfo}.

\subsection{Organization}

In Section~\ref{sec:prelim}, we set up preliminaries, including the process of converting a labeled knotoid diagram to a symmetric ribbon disk and vice versa, along with moves on labeled knotoid diagrams.  In Section~\ref{sec:singular}, we classify singular arc diagrams (knotoid diagrams without crossing information) with two, three, and four crossings.  In Section~\ref{sec:alex}, we determine the sets $\Rr^s_2$ and $\Rr^s_3$ and include Table~\ref{table:r4}.  In Section~\ref{sec:det}, we prove Theorem~\ref{thm:det}, and in Section~\ref{sec:jones}, we prove Theorem~\ref{thm:jones}.  Finally, we use all of our tools to tabulate symmetric ribbon numbers for prime knots up to 12 crossings within the tables in Section~\ref{sec:tab}.  Section~\ref{sec:app} is an appendix in which we complete the proof of Theorem~\ref{thm:r4}.

\subsection{Acknowledgments}

We thank the College of Arts and Sciences at the University of Nebraska-Lincoln for running the three-week course MATH 391: Special Topics in Mathematics in January 2024, during which this project was first conceived and initiated.  NS was supported by UNL's Undergraduate Creative Activity and Research Experience (UCARE) program during summer 2024, and AT was supported by the UCARE program during academic year 2024-2025.  AZ thanks Dimos Goundaroulis for enlightening communications about the tabulation of knotoid diagrams and Jeffrey Meier for helpful conversations about ribbon knots and disks.  AZ was supported in part by NSF grant DMS-2405301 and a Simons Travel Support award.

\section{Preliminaries}\label{sec:prelim}

We work in the smooth category throughout.  A knot $K \subset S^3$ is called \emph{ribbon} if $K$ bounds an immersed disk $\D \subset S^3$ with only ribbon singularities, called a \emph{ribbon disk} for $K$.  The \emph{ribbon number} $r(\D)$ counts the number of ribbon singularities contained in $D$, and the \emph{ribbon number} $r(K)$ of $K$ is defined as
\[ r(K) = \min\{r(\D) : \D \text{ is a ribbon disk for } K\}.\]
If $K$ is non-trivial, then $r(K) \geq 2$ (see~\cite[Remark 2.7]{FMZ}, for instance).  Every ribbon disk $\D$ can be flattened to a projection plane and built from six types of fundamental pieces, shown in Figure~\ref{fig:modular}:  Caps, strips, crossings, singularities, crossings, and junctions.  These pieces were discussed in~\cite{eisermann}, and ~\cite{aceto}, and if a ribbon disk $\D$ is realized as a union of these building blocks, we call $\D$ a \emph{modular ribbon disk}.  

A symmetric union presentation $D^*$, introduced in~\cite{lamm1}, is defined as follows:  First, start with a knot diagram $D$ for a knot $J$.  The diagram $D \# -\!\!D$, which we call the \emph{initial diagram}, has an axis of symmetry $\ell$.  Choose some number of pairwise disjoint disks $\Delta_1,\dots,\Delta_k$ which have reflection symmetry over $\ell$ and such that each disk $\Delta_i$ meets $D \# -\!\!D$ in a pair of arcs that also have reflection symmetry over $\ell$.  Now, replace each $(D \# -\!\!D) \cap \Delta_i$ with a tangle diagram $T_i$ so that all crossings of $T_i$ occur on the axis of symmetry $\ell$.  The resulting diagram $D^*$ is called a \emph{symmetric union presentation} with \emph{partial knot} $J$.  Lamm proved that if $K$ admits a symmetric union presentation with partial knot $J$, then $K$ is a ribbon knot and $\det(K) = (\det(J))^2$~\cite{lamm1}.  See Figure~\ref{fig:symm} for an example of a symmetric union presentation.

Related to modular diagrams, Aceto proved a useful proposition: \vspace{.1cm}

\begin{proposition}~\cite{aceto}\label{prop:symm}
A knot $K \subset S^3$ admits a symmetric union presentation if and only if $K$ bounds a modular ribbon disk $\mathcal{D}$ that contains no crossings or junctions.
\end{proposition}
In light of Proposition~\ref{prop:symm}, a modular ribbon disk $\mathcal{D}$ without crossings or junctions is called \emph{symmetric ribbon disk}, and a knot $K$ bounding a symmetric ribbon disk is called a \emph{symmetric ribbon knot}.  Aceto naturally defined the \emph{symmetric ribbon number} $r_s(K)$ of a symmetric ribbon knot to be
\[ r_s(K) = \min\{r(\D) : \D \text{ is a symmetric ribbon disk for } K\}.\]
Immediately, we have
\begin{lemma}\label{lem:obv}
$r_s(K) \geq r(K)$.
\end{lemma}

Lamm's work in~\cite{lamm1} and~\cite{lamm2} contains a number of symmetric union presentations for knots.  To find upper bounds on the symmetric ribbon numbers of these knots, we describe a method for converting a symmetric union presentation to a labeled knotoid diagram.  A \emph{knotoid diagram} $\K$ is an immersed arc in $S^2$ with only double points, in which each double point is endowed with crossing information.  We call the endpoints of the arc the \emph{aglets} of $\K$.  The $n$ crossings cut $\K$ into $n+1$ connected arcs called \emph{strands}, and a \emph{labeled knotoid diagram} is a knotoid diagram with each strand labeled with an integer.

\begin{lemma}\label{lem:knotoid}
Every symmetric union presentation induces a labeled knotoid diagram, and vice versa.
\end{lemma}

\begin{proof}
Suppose that $K$ admits a symmetric union presentation $D^*$ with partial knot $J$ and initial diagram $D \# -\!\!D$.  Cut $D \# -\!\!D$ along its axis of symmetry $\ell$, yielding two arcs, and interpret one of these arcs as a knotoid diagram $\K$.  Note that $\K$ meets each of the disks $\Delta_1,\dots,\Delta_k$ in the definition of a symmetric union presentation.  For each strand of $\K$, label it with the sum of the (signed) number of crossings contained in all disks $\Delta_1,\dots,\Delta_k$ that the strand meets.


Conversely, we can convert a labeled knotoid diagram $\K$ to a symmetric ribbon disk:  First, we ignore the labeling of $\K$, convert each aglet to a cap, convert each arc to a strip, and convert each crossing to a singularity, as shown below in Figure~\ref{fig:ex2}.  Note that there are two different choices of singularity corresponding to each crossing.  We can make either choice, provided that our choices are \emph{globally consistent}, meaning that we can traverse our diagram from one cap to the other so that every time we pass through a singularity, we travel through two over-crossings (corresponding to an under-crossing in $\K$) or we travel through an over- and under-crossing (corresponding to an over-crossing in $\K$).  If our choices are globally consistent, this process produces a diagram of the form $D \# -\!\!D$.  Finally, we add in twists (with signs) corresponding to the labelings to complete the process.  By Proposition~\ref{prop:symm}, this symmetric ribbon disk induces a symmetric union presentation.
\end{proof}

An example of the process of converting a symmetric union presentation to a labeled knotoid diagram is shown in Figure~\ref{fig:symm}.  An example of the reverse process, converting a labeled knotoid diagram to a symmetric ribbon disk, is carried out in Figure~\ref{fig:ex2}.  By convention, 0-labelings are omitted.

\begin{figure}[h!]
    \centering
    \includegraphics[width=0.8\textwidth]{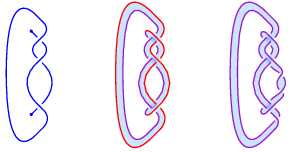}
            \put (-285,35) {$1$}
            \put (-280,77) {$-2$}
            \put (-145,35) {1}
            \put (-141,75) {$-2$}   
     \caption{Converting a labeled knotoid diagram to a symmetric ribbon disk.  In the middle, the red path traverses the diagram in a way that satisfies the global consistency condition (it never passes through a singularity with two under-crossings).  Note further that the knotoid diagram at left is equivalent to the one in Figure~\ref{fig:symm}.}
    \label{fig:ex2}
\end{figure}
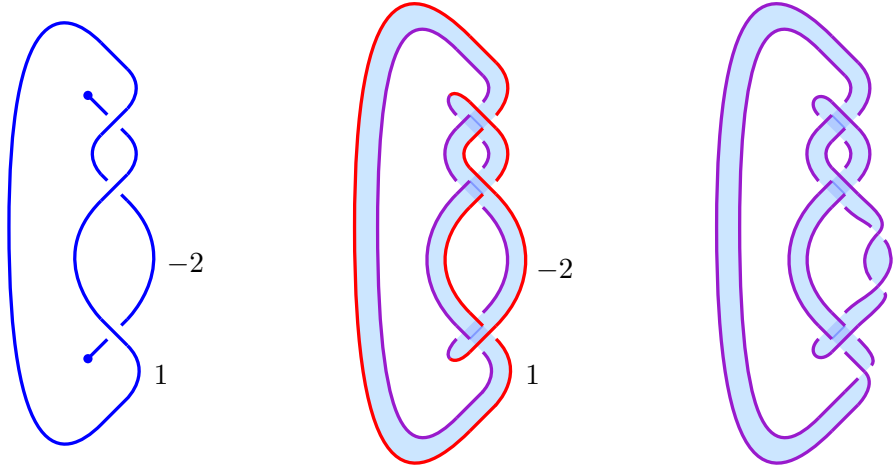

Given a knotoid diagram $\K$, the \emph{closure} of $\K$, denoted $\cl(\K)$, is the knot corresponding to a diagram obtained by connecting the aglets of $\K$ with an arc that contains only over-crossings.

\begin{lemma}\label{lem:det}
If $K$ admits a symmetric union presentation inducing labeled knotoid diagram $\K$, then $\det(K) = (\det(\cl(\K)))^2$.
\end{lemma}

\begin{proof}
This follows immediately from~\cite{lamm1} and the observation that $\cl(\K)$ is a diagram for the partial knot $J$ corresponding to the symmetric union presentation.
\end{proof}


Given two knotoid diagrams $\K'$ and $\K''$, we can form the \emph{knotoid connected sum} $\K' \# \K''$ by removing small disk neighborhood of one aglet from each diagram and gluing along the resulting boundary curves.  Note that the connected sum may depend on our choices of aglets.  Conversely, a knotoid diagram $\K$ is called \emph{composite} if there exists a simple closed curve $c$ meeting $\K$ once and cutting $S^2$ into two disks such that each contains at least one crossing.  In this case, $c$ can be used to realize $\K$ as $\K' \# \K''$.  If the diagram $\K$ is not composite, we say that $\K$ is \emph{prime}.  See Figure~\ref{fig:connect} for an example.

\begin{figure}[h!]
    \centering
    \includegraphics[width=0.8\textwidth]{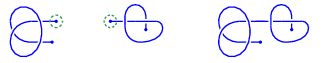}
            \put (-262,35) {$\#$}
            \put (-146,35) {$=$}
                 \caption{An example of the connected sum operation on knotoid diagrams.}
    \label{fig:connect}
\end{figure}

Near the aglet, we note several simplifications.  First, the strand connected to the aglet need not be labeled, since any labeling yields twisting on the induced symmetric ribbon disk that can be eliminated via isotopy, as shown in Figure~\ref{fig:aglet}.  In addition, by construction we have that if a labeled knotoid diagram $\K$ induces a symmetric ribbon disk $\D$, then $r(\D) = c(\K)$, where $c(\K)$ is the crossing number of $\K$.  If the strand leaving the aglet meets an over-crossing of $\K$, we may eliminate the corresponding singularity via isotopy, inducing the move shown in Figure~\ref{fig:move}, which we call an \emph{aglet move}.

\begin{figure}[h!]
    \centering
    \includegraphics[width=0.8\textwidth]{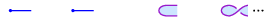}
            \put (-291,11) {$\longleftrightarrow$}
                        \put (-324,20) {$0$}
                   \put (-245,20) {$n$}
            \put (-99,11) {$\longleftrightarrow$}
                 \caption{The labeling of the strand adjacent to the aglet can be changed via isotopy, and so we generally assume this strand is labeled zero.}
    \label{fig:aglet}
\end{figure}

\begin{figure}[h!]
    \centering
    \includegraphics[width=0.8\textwidth]{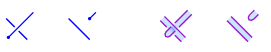}
            \put (-291,27) {$\longleftrightarrow$}
                        \put (-346,47) {$a$}
                         \put (-310,5) {$b$}
                  \put (-272,47) {$a+b$}
            \put (-88,27) {$\longleftrightarrow$}
         \caption{If the strand adjacent to an aglet meets an over-crossing of $\K$, we may eliminate the crossing (and its corresponding singularity) via an aglet move.}
    \label{fig:move}
\end{figure}
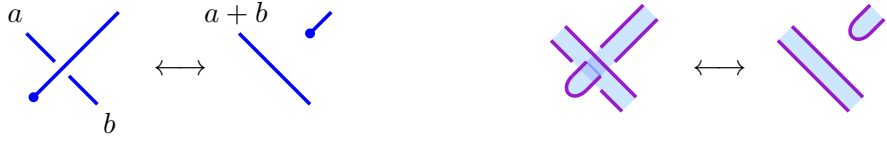

It may be possible to simplify a given labeled knotoid diagram by performing Reidemeister moves.  For example, if a knotoid diagram $\K$ admits an R1 move, then the corresponding symmetric ribbon disk can be simplified by eliminating a singularity, as shown in Figure~\ref{fig:R1}.  For an R2 move, the situation is more complicated.  We can perform an R2 move to simplify the a labeled knotoid diagram if and only if the strand adjacent to both under-crossings has label zero, as shown in Figure~\ref{fig:R2}.  Similarly, an R3 move can be performed only when the strand adjacent to two under-crossings has label zero, as shown in Figure~\ref{fig:R3}.

\begin{figure}[h!]
    \centering
    \includegraphics[width=0.6\textwidth]{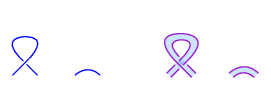}
                \put (-214,33) {$\longleftrightarrow$}
                        \put (-258,15) {$a$}
                         \put (-220,15) {$b$}
                  \put (-189,27) {$a+b$}
            \put (-64,33) {$\longleftrightarrow$}
     \caption{An R1 move on a labeled knotoid diagram translates to a move on the corresponding symmetric ribbon disk.  Labelings change as shown.}
    \label{fig:R1}
\end{figure}

\begin{figure}[h!]
    \centering
    \includegraphics[width=0.6\textwidth]{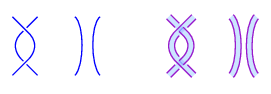}
                    \put (-214,44) {$\longleftrightarrow$}
                        \put (-220,75) {$a$}
                         \put (-220,15) {$b$}
                  \put (-164,40) {$a+b$}
                                    \put (-256,40) {$0$}
            \put (-64,44) {$\longleftrightarrow$}
     \caption{An R2 move on a labeled knotoid diagram translates to a move on the corresponding symmetric ribbon disk.  The strand shown at left must be labeled zero for this move to be possible.}
    \label{fig:R2}
\end{figure}

\begin{figure}[h!]
    \centering
    \includegraphics[width=0.85\textwidth]{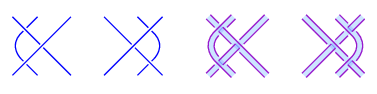}
                        \put (-292,44) {$\longleftrightarrow$}
                                    \put (-364,40) {$0$}
                                    \put (-208,40) {$0$}
            \put (-99,44) {$\longleftrightarrow$}
     \caption{An R3 move on a labeled knotoid diagram translates to a move on the corresponding symmetric ribbon disk.  The strand shown at left must be labeled zero for this move to be possible.}
    \label{fig:R3}
\end{figure}
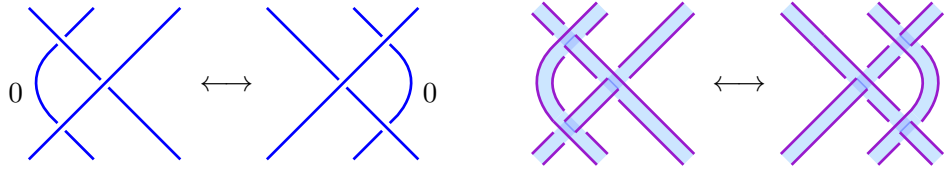

If two labeled knotoid diagrams $\K$ and $\K'$ are related via a sequence of aglet moves, R1 moves, R2 moves, R3 moves, and planar isotopy, we say that $\K$ and $\K'$ are equivalent.  We have demonstrated that if $\K$ and $\K'$ are equivalent, then they induce symmetric ribbon disks $\D$ and $\D'$ such that $\pd \D$ and $\pd \D'$ are isotopic knots.  As an example, the labeled knotoid diagrams in Figures~\ref{fig:symm} and~\ref{fig:ex2} are equivalent.  Note that neither labeled knotoid diagram admits a simplifying R2 move, since the requisite strand has label $-2$ (not zero).  The symmetric ribbon disk in Figure~\ref{fig:ex2} illustrates how nonzero twisting obstructs the isotopy shown in Figure~\ref{fig:R2}.

\begin{remark}
Our construction differs slightly from that in~\cite{lamm2} in our conventions reverse the signs of the crossings in the corresponding knotoid diagrams.  The reason for this change involves the labeling of strands.  With Lamm's conventions, two different labels may be needed on opposite sides of an over-crossing, while a label can be passed through an under-crossing.  For our conventions, twisting can be passed through over-crossings (but not under-crossings), and so it is enough to label each connected strand.  See Figure~\ref{fig:twist}.
\end{remark}

\begin{figure}[h!]
    \centering
    \includegraphics[width=0.5\textwidth]{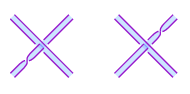}
                            \put (-114,53) {$\simeq$}
     \caption{Twisting on the strip corresponding to an over-crossing of a labeled knotoid diagram can be pushed to either side of the associated singularity.}
    \label{fig:twist}
\end{figure}

We can use labeled knotoid diagrams to obtain upper bounds for symmetric ribbon numbers.  Proposition~\ref{prop:upper} already appears as Lemma 2.1 in~\cite{FMZ}, but we include another proof here using labeled knotoid diagrams.

\begin{proposition}\label{prop:upper}
Suppose that $K$ admits a symmetric union presentation with initial diagram $D \# -\!\!D$ and axis of symmetry $\ell$.  If there are $n$ consecutive over-crossings (or $n$ consecutive under-crossings) in $D$ adjacent to $\ell$, then $K$ bounds a symmetric ribbon disk $\D$ such that $r(\D) = c(D) - n$.
\end{proposition}

\begin{proof}
If there are $n$ consecutive over-crossings adjacent to $\ell$, then the induced labeled knotoid diagram $\K$ admits a total of $n$ simplifying aglet moves, inducing a symmetric ribbon disk $\D$ with $c(D) - n$ singularities, so that $r(\D) = c(D) - n$.  In the case of $n$ consecutive under-crossings, we perform an involution of $S^3$ to convert the under-crossings to over-crossings and repeat the argument.
\end{proof}

The example in Figure~\ref{fig:symm} has two over-crossings adjacent to its axis of symmetry, while $c(D) = 5$.  Thus, aglet moves (carried out in Figure~\ref{fig:ex2}) yield a labeled knotoid diagram with three crossings, in turn yielding a symmetric ribbon disk with ribbon number three.

\section{Classifying reduced singular arc diagrams with at most four crossings}\label{sec:singular}

To tabulate symmetric ribbon numbers, we wish to understand all possible symmetric ribbon disks with ribbon number up to four, and by our work in Section~\ref{sec:prelim}, it suffices to understand labeled knotoid diagrams with at most four crossings.  Although it would be a convenient shortcut to appeal to existing knotoid classifications (see~\cite{dimos}, for example), R2 and R3 moves are not allowed in our setting, and so our notion of knotoid equivalence is different from what usually appears in the literature.  To understand knotoid diagrams within our framework, we first consider a knotoid diagram without crossing information, which we call a \emph{singular arc diagram}, an immersed arc in $S^2$ with double points, considered up to planar isotopy.  A crossing of a singular arc diagrams is said to be \emph{incident} to an aglet if the aglet is connected to the crossing via an embedded arc.  To carry out an exhaustive search for low-crossing singular arc diagrams, we eliminate diagrams that admit obvious simplifications.  A singular arc diagram $\lambda$ is \emph{reduced} if
\begin{enumerate}
\item $\lambda$ does not contain an embedded arc meeting the same crossing in two points,
\item if $c$ is a simple closed curve meeting $\lambda$ in a single point, then $c$ bounds a disk containing an aglet and no crossings, and
\item no crossing is incident to both aglets.
\end{enumerate}
If $\lambda$ is unreduced, then any corresponding labeled knotoid diagram admits a simplifying R1 move (if (1) fails), any corresponding knotoid diagram is not prime (if (2) fails), or any corresponding labeled knotoid diagram admits a simplifying aglet move (if (3) fails).

Given a reduced singular arc diagram $\lambda$, we call a crossing incident to an aglet an \emph{aglet crossing} and any crossing that is not an aglet crossing an \emph{interior crossing}.  In addition, we call the arc leaving an aglet crossing opposite the aglet an \emph{aglet arc}.  Note that a reduced singular arc diagram $\lambda$ will have exactly two aglet crossings, and the two aglet arcs will not connect (otherwise $\lambda$ is an immersion of the disjoint union of an arc and a curve, in which case we say $\lambda$ is \emph{disconnected}).

\begin{lemma}\label{lem:c2}
Up to symmetry, there is one reduced singular arc diagram with two crossings, $2_1$, shown at right in Figure~\ref{fig:c2}.
\end{lemma}

\begin{proof}
Suppose $\lambda$ is a reduced singular arc diagram with $c(\lambda)=2$ and consider the aglet arc of one of the aglet crossings.  The aglet arc must connect to a non-aglet arc, and up to symmetry, this connection can be accomplished in only one way, as shown at left in Figure~\ref{fig:c2}.  Since $\lambda$ is reduced, all remaining arcs must connect the two distinct crossings, and in $S^2$, there is a unique way to make such connections up to isotopy, shown at right in Figure~\ref{fig:c2}.
\end{proof}

\begin{figure}[h!]
    \centering
    \includegraphics[width=0.8\textwidth]{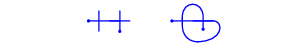}
                                    \put (-127,-10) {$2_1$}
     \caption{At left, an arc must connect the two aglet crossings of a reduced singular arc diagram $\lambda$ with $c(\lambda) = 2$.  At right, completing the diagram to get $\lambda$, which we call $2_1$.}
    \label{fig:c2}
\end{figure}

Moving to crossing number three is only slightly more complicated.

\begin{lemma}\label{lem:c3}
Up to symmetry, there are two reduced singular arc diagrams with three crossings, $3_1$ and $3_2$, shown at bottom left and bottom middle in Figure~\ref{fig:c3}.
\end{lemma}

\begin{proof}
Let $\lambda$ be a reduced singular arc diagram with $c(\lambda) = 3$.  A total of five arcs connect the three crossings, and since the interior crossing of $\lambda$ must meet exactly four of these arcs, there must be one arc connecting the aglet crossings (and this arc cannot be the aglet arc for both crossings).  Up to symmetry, there are three possible ways to connect the two aglet crossings via an arc, shown at top in Figure~\ref{fig:c3}.  Since all four remaining arcs must connect to the interior crossing, each configuration at top can be uniquely completed (up to isotopy) to a possible singular arc diagram $\lambda$, shown at bottom in Figure~\ref{fig:c3}.  However, the bottom right diagram is disconnected, leaving only the two possibilities at bottom left and middle.
\end{proof}

\begin{figure}[h!]
    \centering
    \includegraphics[width=0.75\textwidth]{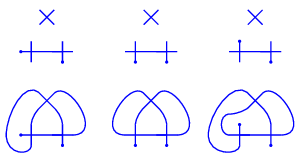}
                                    \put (-279,0) {$3_1$}
                                \put (-166,0) {$3_2$}
     \caption{At top, possible connections between aglet crossings.  At bottom, completing the diagrams.  Note that only the bottom left ($3_1$) and middle ($3_2$) are valid, since the bottom right diagram is disconnected.}
    \label{fig:c3}
\end{figure}

The case that $c(\lambda) = 4$ requires significantly more work.

\begin{proposition}\label{prop:c4}
Up to symmetry, there are eight reduced singular arc diagrams with four crossings, $4_1,4_2,4_3,4_4,4_5,4_6,4_7$, and $4_8$, shown in Figure~\ref{fig:c4}.
\end{proposition}

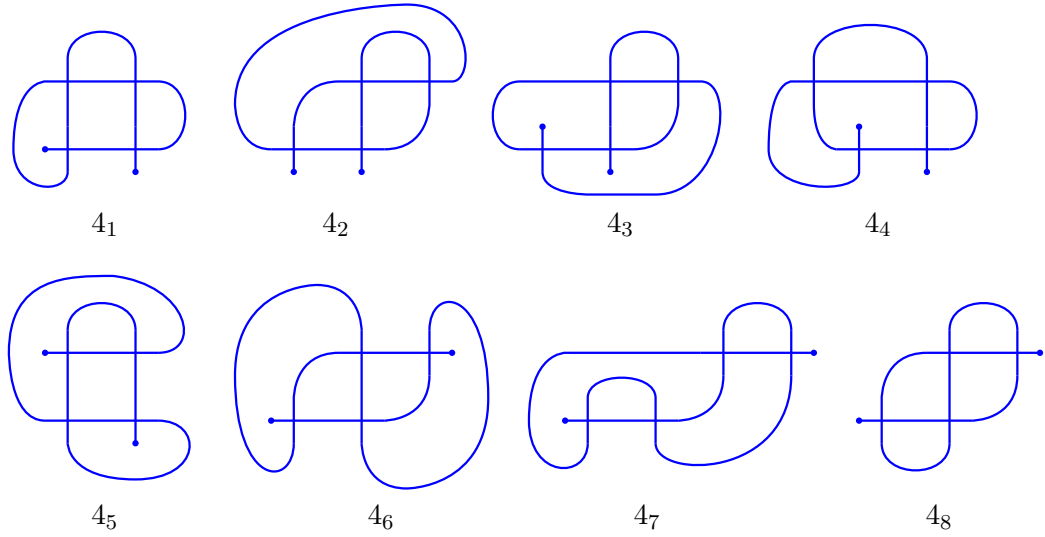
\begin{figure}[h!]
    \centering
    \includegraphics[width=.95\textwidth]{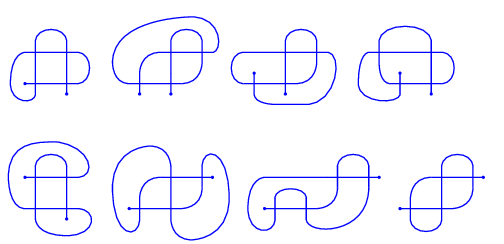}
                                \put (-374,100) {$4_1$}
                            \put (-287,100) {$4_2$}
                            \put (-180,100) {$4_3$}
                            \put (-83,100) {$4_4$}
                            \put (-374,-10) {$4_5$}
                            \put (-270,-10) {$4_6$}
                            \put (-170,-10) {$4_7$}
                            \put (-60,-10) {$4_8$}
     \caption{The eight possible reduced singular arc diagrams with four crossings.}
    \label{fig:c4}
\end{figure}

\begin{proof}
Let $\lambda$ be a reduced singular arc diagram with $c(\lambda) = 4$, so that $\lambda$ contains two aglet crossings and two interior crossings, which are connected by a total of seven arcs.  We break the classification of $\lambda$ into cases, depending on whether there are zero, one, or two arcs connecting the aglet crossings of $\lambda$.

\emph{Case A}:  Suppose that $\lambda$ contains two arcs connecting the aglet crossings.  Then two arcs must connect aglet crossings to interior crossings, which leaves three arcs connecting interior crossings.  However, it is impossible to connect interior crossings with three arcs without creating a closed loop, resulting in $\lambda$ being disconnected, a contradiction.

\emph{Case B}:  Suppose that $\lambda$ contains one arc connecting the aglet crossings.  The three possible configurations for this connection are shown at top in Figure~\ref{fig:c3}.  In each of these three configurations, we know that each of the six remaining arcs must be connected to at least one interior crossing, and so we connect one of the aglet arcs to an interior crossing, as shown in Figure~\ref{fig:c4-10}.  Of the five remaining arcs, four must connect to the additional interior crossing and one does not, with both endpoints connected to one of the structures shown in Figure~\ref{fig:c4-10}.  For each of the three structures, we first add this additional arc without introducing extra components, separating the remaining connections, forcing $\lambda$ to be unreduced, or adding a second arc connecting aglet crossings, after which there is a unique way to add the second interior crossing and the four arcs connected to it.

\begin{figure}[h!]
    \centering
    \includegraphics[width=0.75\textwidth]{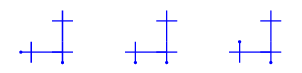}
     \caption{The three structures from Figure~\ref{fig:c3} with an additional interior crossing attached to an aglet arc.}
    \label{fig:c4-10}
\end{figure}

For the first structure, there are two admissible ways to add an arc, shown at top in Figure~\ref{fig:c4-1a}.  Once the additional arc is added, we can uniquely add the other interior crossing and four remaining arcs, shown at bottom in Figure~\ref{fig:c4-1a}.  Note that one of the completions is disconnected.

\begin{figure}[h!]
    \centering
    \includegraphics[width=0.75\textwidth]{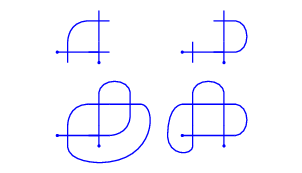}
     \caption{At top, the addition of one arc to the first structure.  At bottom, completing the diagrams by including another interior crossing and four more arcs.  The completion at bottom left is disconnected, so we disregard it.}
    \label{fig:c4-1a}
\end{figure}

With the second structure, there are two valid ways to add this additional arc, shown at top in Figure~\ref{fig:c4-1b}, and then each diagram can be uniquely completed with an additional interior crossing and four connecting arcs, shown at bottom in Figure~\ref{fig:c4-1b}.  As above, one of these completions is disconnected.  With the third configuration, there are two admissible ways to add the additional arc, shown at top in Figure~\ref{fig:c4-1c}, and both of the completions are valid, shown at bottom in the same figure.

\begin{figure}[h!]
    \centering
    \includegraphics[width=0.75\textwidth]{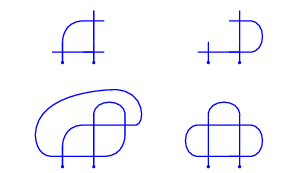}
     \caption{At top, the addition of one arc to the second structure.  At bottom, the completions, where the one at right is disconnected.}
    \label{fig:c4-1b}
\end{figure}

\begin{figure}[h!]
    \centering
    \includegraphics[width=0.75\textwidth]{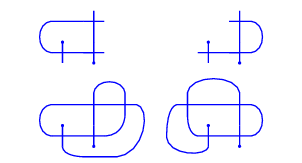}
     \caption{At top, the addition of one arc to the third structure.  At bottom, both completions are valid.}
    \label{fig:c4-1c}
\end{figure}

\emph{Case C}:  Suppose that $\lambda$ contains no arcs connecting the aglet crossings.  Then all arcs connected to an aglet crossing are also connected to an interior crossing.  There are several sub-cases to consider.  First, suppose that the aglet arcs are connected to the same interior crossing.  Unless $\lambda$ is disconnected, there is a unique configuration, shown at left in Figure~\ref{fig:c4-0a}.  Of the five remaining arcs, four are connected to the other interior crossing, and so one more connects an aglet crossing to the first interior crossing.  Up to symmetry, there is only one way to add this arc without disconnecting the other connections, adding another arc connecting aglet crossings, or forcing $\lambda$ to be unreduced, shown at middle in Figure~\ref{fig:c4-0a}, and then the second interior crossing and connecting arcs are completed uniquely up to isotopy, shown at right in Figure~\ref{fig:c4-0a}.

\begin{figure}[h!]
    \centering
    \includegraphics[width=0.75\textwidth]{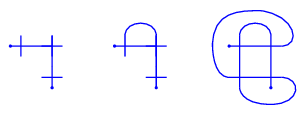}
     \caption{At left, the two aglet arcs connect to the same interior crossing.  At middle, one additional arc is added to this configuration.  At right, the unique completion is valid.}
    \label{fig:c4-0a}
\end{figure}

Next, suppose that the aglet arcs are connected to different interior crossings.  Then five more arcs must be added to the diagram at left in Figure~\ref{fig:c4-0b}.  Suppose first that no additional arcs connect the top aglet crossing to the top interior crossing.  Then two arcs connect the top aglet crossing to the bottom interior crossing, and two arcs connect the bottom aglet crossing to the top interior crossing.  Up to symmetry, the diagram is uniquely determined, as shown at right in Figure~\ref{fig:c4-0b}.

\begin{figure}[h!]
    \centering
    \includegraphics[width=0.75\textwidth]{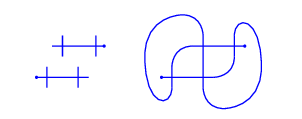}
     \caption{At left, the two aglet arcs connect to different interior crossings.  At right, the unique completion (up to symmetry) with all arcs connecting the top structure to the bottom structure and no arcs connecting the aglet crossings.}
    \label{fig:c4-0b}
\end{figure}

Now, suppose that one additional arc connects the top aglet crossing to the top interior crossing.  Then one additional arc connects the bottom aglet crossing to the bottom interior crossing.  There are two possible configurations, shown at left in Figures~\ref{fig:c4-0c} and~\ref{fig:c4-0d}.  For each configuration, the remaining three arcs must connect aglet crossings to interior crossings.  For the first configuration, there is a unique way to make these connections, up to symmetry, shown at right in Figure~\ref{fig:c4-0c}.  For the second configuration, there are two possible ways to make these connections, shown at center and left in Figure~\ref{fig:c4-0d}, but one (at right) is disconnected.

\begin{figure}[h!]
    \centering
    \includegraphics[width=0.75\textwidth]{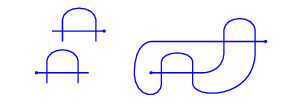}
     \caption{At left, one configuration in which one additional arc is attached to the top structure and one is attached to the bottom structure.  At right, the unique completion up to symmetry.}
    \label{fig:c4-0c}
\end{figure}

\begin{figure}[h!]
    \centering
    \includegraphics[width=0.75\textwidth]{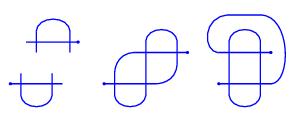}
     \caption{At left, another configuration in which one additional arc is attached to the top structure and one is attached to the bottom structure.  At center and right, the two possible completions.  The completion at right is disconnected.}
    \label{fig:c4-0d}
\end{figure}

Finally, suppose that two additional arcs connect the top aglet crossing to the top interior crossing.  Then the single arc connecting the top structure to the bottom structure will force $\lambda$ to not be reduced.  One example is shown in~\ref{fig:c4-0e}.  This completes the proof.

\begin{figure}[h!]
    \centering
    \includegraphics[width=0.75\textwidth]{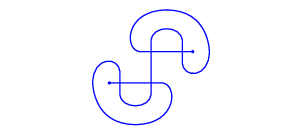}
     \caption{If only one arc connects the top structure to the bottom structure, then the diagram $\lambda$ is not reduced.}
    \label{fig:c4-0e}
\end{figure}

\end{proof}

\section{Alexander polynomials and the sets $\Rr^s_n$}\label{sec:alex}

In~\cite{FMZ} and~\cite{polymath}, the authors used Alexander polynomials to derive new lower bounds for ribbon numbers.  In particular, they defined the set
\[ \Rr_n = \{\Delta_K(t) : r(K) \leq n\}\]
and proved that
\begin{proposition}
For every $n$, the set $\Rr_n$ is finite and computable.
\end{proposition}
In~\cite{FMZ}, the authors determined the sets $\Rr_2$ and $\Rr_3$, where $|\Rr_2| = 3$ and $|\Rr_3| = 10$.  This work was extended in~\cite{polymath}, in which the authors computed $\Rr_4$, where $|\Rr_4| = 56$.

Recall the oriented Skein relation for the Alexander and Jones polynomials for oriented links $L_+$, $L_-$, and $L_0$ that differ within a single crossing:

\[ \Delta_{L_+}(t) - \Delta_{L_-}(t) + (t^{-1/2} - t^{1/2})L_0(t) = 0\]
and
\[t^{-1}V_{L_+}(t) - tV_{L_-}(t) + (t^{-1/2}-t^{1/2})V_{L_0}(t) = 0.\]

In what follows, we consistently choose this crossing within a ribbon disk as shown in Figure~\ref{fig:change}.

\begin{figure}[h!]
    \centering
    \includegraphics[width=0.7\textwidth]{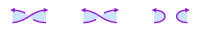}
                            \put (-265,-3) {$L_+$}
                             \put (-157,-3) {$L_-$}
                             \put (-49,-3) {$L_0$}
     \caption{Local pictures of changing or resolving a crossing in a ribbon disk used in the proof of Lemma~\ref{lem:change}.}
    \label{fig:change}
\end{figure}

Here, we define an analogously collection of sets of Alexander polynomials for symmetric ribbon knots.  However, the structure of symmetric ribbon disks is more rigid, and so we define
\[ \Rr^s_n = \{\Delta_K(t) : r_s(K) = n \text{ and $K$ is prime}\}.\]
It follows immediately from the definition that $\Rr^s_n \subset \Rr_n$, and so the set $\Rr^s_n$ is finite.  We will see that this containment is proper for $n=2,3,4$.  Note that if $r_s(K) = n$, then $K$ bounds a symmetric ribbon disk corresponding to a labeled monoid diagram $\K$ with $n$ crossings.  In turn $\K$ can be associated with a singular arc diagram.  We use the following lemma.

\begin{lemma}\label{lem:change}
Suppose $K$ and $K'$ are symmetric ribbon knots bounding disks corresponding to labeled monoid diagrams $\K$ and $\K'$, respectively, which are identical but with potentially different labelings.  If all labelings agree mod 2, then $\Delta_K(t) = \Delta_{K'}(t)$.
\end{lemma}

\begin{proof}
Consider the Skein relation above and the local pictures of a crossing change performed on a symmetric ribbon disk as shown in Figure~\ref{fig:change}, where $L_+$ and $L_-$ are two symmetric ribbon knots corresponding to labeled monoid diagrams with identical labels except for one label differing by two, and $L_0$ is a two-component ribbon link, where $\Delta_{L_0}(t) = 0$ (see, for instance,~\cite{eisermann}).  It follows immediately that $\Delta_{L_+}(t) = \Delta_{L_-}(t)$.
\end{proof}

As a consequence, in order to determine all possible Alexander polynomials for prime knots with symmetric ribbon number $n$, we need only consider all possible reduced singular arc diagrams with $n$ crossings, along with all possible crossing information and mod 2 labelings.  We call such a diagram a \emph{mod 2 labeled knotoid diagram}.  Moreover, we can assume that all aglet arcs meet an aglet crossing as an under-crossing (or else we could perform a simplifying aglet move), and so need only consider crossing information for interior crossings.  Once we find all possible mod 2 labeled knotoid diagrams, we can use SnapPy~\cite{snappy} within Sage~\cite{sage} to compute the corresponding Alexander polynomials.  We carry this out for $n=2,3,4$ below.

For $n=2$, there is only one reduced singular arc diagram with two crossings, $2_1$, shown in Figure~\ref{fig:c2}.

\begin{lemma}\label{lem:r2}
$\Rr^s_2 = \{1-2t+3t^2-2t^3+t^4, \, 2-5t+2t^2\}$.
\end{lemma}

\begin{proof}
The reduced singular arc diagram $2_1$ gives rise to two possible mod 2 labeled knotoid diagrams, and the corresponding Alexander polynomials are as indicated.
\end{proof}

For $n=3$, there are two reduced singular arc diagrams with three crossings, $3_1$ and $3_2$, shown in Figure~\ref{fig:c3}.  We make our work easier by introducing a new move.  A \emph{leaf isotopy} (terminology taken from a similar move in~\cite{polymath}) on a labeled knotoid diagram is the combination of moves shown in Figure~\ref{fig:leaf}.

\begin{figure}[h!]
    \centering
    \includegraphics[width=0.8\textwidth]{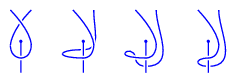}
                                        \put (-219,54) {$0$}
                                    \put (-149,18) {$0$}
     \caption{A leaf isotopy is a combination of an aglet move, an R3 move, and an R1 move as shown.  The only relevant labeling is the zero-labeling of the under-strand in the R3 move occurring between the two middle frames.}
    \label{fig:leaf}
\end{figure}

\begin{lemma}\label{lem:r3}
\begin{eqnarray*}
\Rr^s_3 &=& \{1, \\
& & 1 - t - t^2 + 3t^3 - t^4 - t^5 + t^6, \\
& & 1 - 3t^2 + t^4, \\
& & 1 - 3t + 5t^2 - 7t^3 + 5t^4 - 3t^5 + t^6, \\
& & 1 - 6t + 11t^2 - 6t^3 + t^4, \\
& & 2 - 6t + 9t^2 - 6t^3 + 2t^4\}.
\end{eqnarray*}
\end{lemma}

\begin{proof}
Note that $3_2$ has a reflection symmetry, and so both choices for the interior crossing will produce identical Alexander polynomials.  In addition, one of the choices for the interior crossing of $3_1$ can be transformed a leaf isotopy into $3_2$ as shown in Figure~\ref{fig:moves3}.  Thus, we need only consider $3_1$.  There are two choices for the interior crossing and two choices for the labels of a mod 2 labeled knotoid diagram associated to $3_1$, yielding eight possibilities.  We compute the Alexander polynomials of the corresponding knots, for a total of six distinct polynomials as indicated.
\end{proof}

\begin{figure}[h!]
    \centering
    \includegraphics[width=0.6\textwidth]{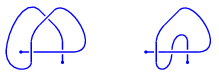}
     \caption{Converting a labeled knotoid diagram coming from $\lambda_1$ to one coming from $\lambda_2$ via a leaf isotopy.}
    \label{fig:moves3}
\end{figure}

Next, we turn to $\Rr^s_4$.  Theorem~\ref{thm:r4}, stated in Section~\ref{sec:intro}, asserts that the 27 elements of $\Rr^s_4$ are given in Table~\ref{table:r4}.  The proof is similar to the proofs of Lemmas~\ref{lem:r2} and~\ref{lem:r3} above, but we relegate its proof to the appendix in Section~\ref{sec:app} due to its length.

\begin{table}[!ht]
    \centering
    \begin{tabular}{|l|l|l|l|}
    \hline
        Det & Alexander Polynomial \\ \hline
        1 & $1$ \\ 
        1 & $1-3t^2+t^4$ \\ 
        1 & $1-t-t^2+3t^3-t^4-t^5+t^6$ \\ 
        1 & $1-2t+t^2 +2t^3 -5t^4 +2t^5 +t^6 -2t^7 +t^8$ \\ 
        1 & $1-2t^2 +3t^4 -2t^6 +t^8$ \\ 
        1 & $1 -3t-t^2 +7t^3 -t^4 -3t^5 +t^6$ \\
        1 & $1-3t+2t^2+3t^3-7t^4+3t^5+2t^6-3t^7+t^8$ \\
        1 & $2 -3t-2t^2 +7t^3 -2t^4 -3t^5 +2t^6$ \\ 
        1 & $2-5t^2 +2t^4$ \\ \hline
        9 & $1-t-t^3 +3t^4 -t^5 -t^7 +t^8$ \\ 
        9 & $1-t+t^2 -3t^3 +t^4 -t^5 +t^6$ \\ 
        9 & $1 - t - 3t^2 + 7t^3 - 3t^4 - t^5 + t^6$ \\ 
        9 & $1-2t+3t^2-2t^3+t^4$ \\ 
        9 & $1-2t+4t^3 -7t^4 +4t^5 - 2t^7 + t^8$ \\
        9 & $2-5t+2t^2$ \\ \hline
        25 & $1-2t+3t^2 -4t^3 +5t^4 - 4t^5 + 3t^6 - 2t^7 + t^8$ \\ 
        25 & $1 - 3t + 5t^2 - 7t^3 + 5t^4 - 3t^5 + t^6$ \\ 
        25 & $1 - 6t + 11t^2 - 6t^3 + t^4$ \\ 
        25 & $2 - 6t + 9t^2 - 6t^3 + 2t^4$ \\         
        25 & $6-13t+6t^2$ \\ \hline
        49 & $1-3t+6t^2 -9t^3 +11t^4 -9t^5 +6t^6 -3t^7 +t^8$ \\ 
        49 & $1-4t+6t^2 -8t^3 +11t^4 -8t^5 +6t^6 -4t^7 +t^8$ \\
        49 & $1-5t+11t^2 -15t^3 +11t^4 - 5t^6 + t^6$ \\ 
        49 & $2-6t+10t^2 -13t^3 +10t^4 -6t^5 +2t^6$ \\ 
        49 & $2-12t+21t^2 -12t^3 +2t^4$ \\ 
        49 & $3-12t+19t^2 -12t^3 +3t^4$ \\ 
        49 & $4-12t+17t^2 -12t^3 +4t^4$ \\ \hline
    \end{tabular}
	\caption{Elements of $\Rr^s_4$, ordered by determinant}
	\label{table:r4}
\end{table}

\section{Bounds from determinants}\label{sec:det}

Recall that the determinant $\det(K)$ of a knot $K$ is given by $\det(K) = |\Delta_K(-1)|$.  In~\cite{FMZ}, the authors proved that for a ribbon knot $K$,
\[ \det(K) \leq (2^{r(K)} - 1)^2.\]
For low-complexity ribbon knots, this implies
\begin{itemize}
\item If $r(K) = 2$, then $\det(K) \leq 9$.
\item If $r(K) = 3$, then $\det(K) \leq 49$.
\item If $r(K) = 4$, then $\det(K) \leq 225$, and so on.
\end{itemize}

By evaluating all polynomials in $\Rr^s_2$, $\Rr^s_3$, and $\Rr^s_4$ at $-1$, we obtain the following immediate corollary of Lemmas~\ref{lem:r2} and~\ref{lem:r3} and Theorem~\ref{thm:r4}.

\begin{corollary}\label{cor:det}
Suppose $K$ is a prime symmetric ribbon knot.
\begin{itemize}
\item If $r_s(K) = 2$, then $\det(K) = 9$
\item If $r_s(K) = 3$, then $\det(K) = 1$ or $25$.
\item If $r_s(K) = 4$, then $\det(K) \leq 49$.
\end{itemize}
\end{corollary}

In this section, we obtain a similar inequality for knots such that $r_s(K) = 5$.  Absent a version of Theorem~\ref{thm:r4} in this case, we take a different approach.  Given a knot diagram $D$, define the \emph{maximal overpass length} $\ell(D)$ to be the largest number of consecutive over-crossings contained in a single strand of $D$.  This quantity is connected to knotoid diagrams via the next lemma.  Recall that given a knotoid diagram $\K$, the \emph{closure} $\cl(\K)$ is the knot obtained by connecting the aglets with an arc that contains only over-crossings.

\begin{lemma}\label{lem:overpass}
Suppose $\K$ is a knotoid diagram with closure $\cl(\K)$.  Then there is a diagram $D$ for $\cl(\K)$ such that
\[ c(D) - \ell(D) \leq c(\K).\]
\end{lemma}

\begin{proof}
Let $D$ be a diagram for $\cl(\K)$ obtained by connecting the aglets of $\K$ with an overpass arc, where $D$ contains $n$ crossings, so that $n = c(D) - c(\K)$.  It follows that $n \leq \ell(D)$, so that $c(D) - c(\K) \leq \ell(D)$.  The desired statement follows.
\end{proof}

In~\cite{kidwell}, Kidwell connected $c(D) - \ell(D)$ to the degree of the $Q$-polynomial $Q_K(z)$ of the corresponding knot $K$, where the $Q$-polynomial is a specialization $Q_K(z) = F(1,z)$ of the Kauffman polynomial $F_K(a,z)$.  Translated to the Kauffman polynomial, Kidwell proved

\begin{theorem}\cite{kidwell}\label{thm:kid}
Let $D$ be a diagram of a knot $K$.  Then
\[ \deg_z(F_K(a,z)) \leq c(D) - \ell(D).\]
\end{theorem}

Combining Lemma~\ref{lem:overpass} and Theorem~\ref{thm:kid}, we have

\begin{corollary}\label{cor:f}
Let $\K$ be a knotoid diagram.  Then
\[ \deg_z(F_{\cl(\K)}(a,z)) \leq c(\K).\]
\end{corollary}

Another useful tool, also proved by Kidwell as an appendix in~\cite{stoimenow}, is

\begin{theorem}\cite[Theorem 2.1]{stoimenow}\label{thm:32}
Suppose $\K$ is a knotoid diagram.  Then there is a diagram $D$ for $\cl(\K)$ such that
\[ c(D) \leq \frac{3}{2} \cdot c(\K).\]
\end{theorem}

A knot diagram $D$ is called \emph{prime} if there does not exist a simple closed curve $c$ meeting $D$ in two points such that there are crossings of $D$ on both sides of $c$.  The \emph{breadth} $\text{breadth}(V_K(t))$ of the Jones polynomial $V_K(t)$ is the difference between the largest and smallest degrees appearing in $V_K(t)$.  Murasugi and Thistlethwaite proved that if $D$ is a prime, non-alternating diagram for a knot $K$, then $c(D) > \text{breadth}(V_K(t))$~\cite{murasugi, thistlethwaite}.

We require one more technical, highly specialized lemma.

\begin{lemma}\label{lem:connect}
Suppose $\K$ is a knotoid diagram such that $\cl(\K)$ is the connected sum of the trefoil and figure eight knot.  Then either $\K$ is not prime, or $c(\K) > 5$.
\end{lemma}

\begin{proof}
Suppose $\K$ is a knotoid diagram such that $\cl(\K) = K_1 \# K_2$, where $K_1$ is the trefoil and $K_2$ is the figure eight knot, and let $J = \cl(\K)$.  Using the Kauffman polynomial data from the KnotInfo database~\cite{knotinfo} along with the fact that the Kaufmann polynomial is multiplicative under connected sum, we have that $\deg_z(F_{J}(a,z)) = 5$, so that $c(\K) \geq 5$ by Corollary~\ref{cor:f}.  Suppose $c(\K) = 5$, and note that $\K$ cannot admit a simplifying R1 move.  By Theorem~\ref{thm:32}, there is a diagram $D$ for $J$ such that $c(D) \leq 7$, and since $c(J) = 7$, we must have $c(D) = 7$.  Let $\A$ denote the arc used to obtain $D$ from $\K$, so that $\A$ contains two over-crossings, implying that $D$ is non-alternating.  By Murasugi and Thistlethwaite's Theorem along with the fact that $\text{breadth}(V_J(t)) = 7$, it follows that $D$ is not prime, so there exists a curve $c$ meeting $D$ in two points, with crossings on either side of $c$.

It follows that we can express $D$ as $D_1 \# D_2$, and since $D$ is a minimal-crossing diagram, we may assume without loss of generality that $D_1$ is a 3-crossing diagram for $K_1$ and $D_2$ is a 4-crossing diagram for $K_2$.  If the arc $\A$ meets $c$ in two points, then the knotoid diagram $\K$ is not connected, a contradiction.  If $\A$ meets $c$ in one point, then the knotoid diagram $\K$ also meets $c$ in one point.  As $c(D_1) =3$, $c(D_2) = 4$, and $\A$ contains two crossings, it follows that $\K$ contains crossings on either side of $c$, so that $\K$ is not prime (since $\K$ cannot admit a simplifying R1 move).  Finally, if $\A$ is disjoint from $c$, then $\A$ is contained entirely within $D_1$ or $D_2$.

If $\A$ is contained entirely within $D_1$, then so are both aglets of $\K$.  By cutting $\K$ open along $c$ and gluing in an arc of $c$, we can obtain a knotoid $\K'$ such that $c(\K') = 1$ and $\cl(\K') = K_1$, contradicting Corollary~\ref{cor:f}.  See Figure~\ref{fig:cut}.  On the other hand, if $\A$ is contained entirely within $D_2$, cutting $\K$ open along $c$ and gluing in an arc of $c$ yields a knotoid $\K''$ such that $c(\K'') = 2$ and $\cl(\K'') = K_2$, another contradiction to Corollary~\ref{cor:f}.  We conclude that either $c(\K) > 5$ or $\K$ is not prime.
\end{proof}

\begin{figure}[h!]
    \centering
    \includegraphics[width=0.7\textwidth]{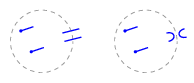}
                                        \put (-275,110) {$c$}
                                        \put (-110,110) {$c$}
                                        \put (-76,34) {$\kappa'$}
     \caption{If both aglets of $\K$ are inside of $D_1$, we surger $\K$ along $c$ to obtain a knotoid $\K'$ within $D_1$.}
    \label{fig:cut}
\end{figure}

We have assembled all of the necessary ingredients to prove our next bound on symmetric ribbon number, Theorem~\ref{thm:det}, which asserts that if $r_s(K) = 5$, then $\det(K) \leq 169$.

\begin{proof}[Proof of Theorem~\ref{thm:det}]
Suppose $r_s(K) = 5$.  Then $K$ bounds a symmetric ribbon disk with associated labeled knotoid diagram $\K$ such that $c(\K) = 5$.  Let $J  = \cl(\K)$.  By Corollary~\ref{cor:f} and Theorem~\ref{thm:32}, we have
\[ \deg_z(F_{J}(a,z)) \leq 5 \text{ and } c(J) \leq 7.\]
If $J$ is prime, then using the Kauffman polynomial and determinant data from the KnotInfo database~\cite{knotinfo}, we see that $J$ must satisfy $c(J) \leq 6$, since all 7-crossing knots have Kauffman polynomials with $z$-degree six.  The maximum possible determinant of such $J$ is 13, so we have $\det(J) \leq 13$, and by Lemma~\ref{lem:det}, it follows that $\det(K) \leq 169$.

On the other hand, suppose $J$ is composite.  Since crossing number is additive under connected sum for alternating knots, it follows that $J = J_1 \# J_2$, where $c(J_1) \leq 3$ and $c(J_2) \leq 4$.  If both $J_1$ and $J_2$ are trefoils, then $\det(K) \leq 81$.   Otherwise, $J_1$ is the trefoil and $J_2$ is the figure eight knot.  By Lemma~\ref{lem:connect} and the assumption that $c(\K) = 5$, we have that $\K$ is not prime.  But this implies $K$ is not prime, a contradiction.
\end{proof}

\begin{remark}\label{rmk:det} 
The proof of Lemma~\ref{lem:connect} can be adapted to knots summands with greater crossing numbers, so that similar bounds can be obtained for larger symmetric ribbon numbers.  Using the KnotInfo data up to 12 crossings, we can generalize the above to obtain the following:
\begin{itemize}
\item If $K$ is a prime symmetric ribbon knot such that $r_s(K) = 6$, then $\det(K) \leq 21^2$.
\item If $K$ is a prime symmetric ribbon knot such that $r_s(K) = 7$, then $\det(K) \leq 75^2$.
\item If $K$ is a prime symmetric ribbon knot such that $r_s(K) = 8$, then $\det(K) \leq 121^2$.
\end{itemize}
We do not, however, use these results in our tabulation below.  Consider the sequence $\{3, 5, 7, 13, 21, 75, 121,\dots\}$.  Squares of its entries show up as upper bounds in Corollary~\ref{cor:det}, Theorem~\ref{thm:det}, and the inequalities above.  This sequence coincides with the maximal determinant of a knot with $n$ crossings, starting with $n=3$~\cite{stoimenow}.  In general, we conjecture that the maximum determinant of a prime symmetric ribbon knot $K$ with $r_s(K) = n$ is the square of the maximum determinant of a prime knot $J$ with $c(J) = n+1$.  
\end{remark}

\section{Jones polynomials and symmetric ribbon knots}\label{sec:jones}

In this section, we determine all possible Jones polynomials for symmetric ribbon knots $K$ with $r_s(K) = 2$ or $r_s(K) = 3$ and $\det(K) = 25$.  Although these families are not finite, as in the case of $\Rr^s_n$, they are manageable enough to extract additional lower bounds on symmetric ribbon numbers.  Recall the oriented skein relation for the Jones polynomial from Section~\ref{sec:alex}.  We will let $U_2$ denote the two-component unlink, noting that $V_{U_2}(t) = -t^{-1/2} - t^{1/2}$.  Computations are performed using SnapPy~\cite{snappy} inside of Sage~\cite{sage}.  Observe that SnapPy within Sage uses the variable $q$ for the Jones polynomial, and we convert to the variable $t$ via the usual substitution $t=q^2$ (conversely, $q = -t^{1/2}$).

\begin{proposition}\label{prop:jones2}
Suppose $K$ is a knot with $r_s(K) = 2$.  There is a family $\{K_n\}$ of symmetric ribbon knots such that $K \in \{K_n\}$ and
\[ V_{K_n}(t) = (-t)^n(-t^{-3}+t^{-2}-t^{-1}+2-t + t^2 - t^{3})  + 1.\]
\end{proposition}

\begin{proof}
By Lemma~\ref{lem:c2}, if $K$ satisfies $r_s(K) = 2$, then $K$ bounds a symmetric ribbon disk corresponding to the labeled knotoid diagram $\K_n$ in Figure~\ref{fig:rs2}.  Let $K_n$ denote the knot arising from the label $n$.

\begin{figure}[h!]
    \centering
    \includegraphics[width=0.3\textwidth]{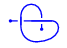}
                                        \put (-15,20) {$n$}
     \caption{The only labeled knotoid diagram $\K_n$ for knots $K$ with $r_s(K) = 2$.}
    \label{fig:rs2}
\end{figure}

There are two cases to consider, whether $n$ positive or negative.  The proof proceeds by induction, and we verify the base cases for $K_0$ and $K_1$ by direction computation.  First, let $n > 1$ and suppose the proposition holds for $K_{n-2}$.  Choosing $L_+$, $L_-$, and $L_0$ as in Figure~\ref{fig:change}, we have that $L_+ = K_{n}$, $L_- = K_{n-2}$, and $L_0 = U_2$.  Let $f(t) = -t^{-3}+t^{-2}-t^{-1}+2-t + t^2 - t^{-3}$.  Applying the skein relation, we have
\begin{eqnarray*}
V_{K_{n}}(t) &=& t^2 \, V_{K_{n-2}}(t) - (t^{1/2}-t^{3/2})V_{U_2}(t) \\
&=& t^2((-t)^{n-2} \cdot f(t)  + 1) - (t^{1/2}-t^{3/2})(-t^{-1/2}-t^{1/2}) \\
&=& (-t)^{n} \cdot f(t) + t^2 - t^2 +1 \\
&=& (-t)^{n} \cdot f(t) + 1.
\end{eqnarray*} 
On the other hand, if $n < 0$, we suppose that the proposition holds for $K_{n+2}$ and apply the skein relation with $K_{n} = L_-$, $K_{n+2} = L_+$, and $L_0 = U_2$.  We have
\begin{eqnarray*}
V_{K_{n}}(t) &=& t^{-2} \, V_{K_{n+2}}(t) + (t^{-3/2}-t^{-1/2})V_{U_2}(t) \\
&=& t^{-2}((-t)^{n+2} \cdot f(t)  + 1) + (t^{-3/2}-t^{-1/2})(-t^{-1/2}-t^{1/2}) \\
&=& (-t)^{n} \cdot f(t) + t^{-2} - t^{-2} +1 \\
&=& (-t)^{n} \cdot f(t) + 1.
\end{eqnarray*} 
\end{proof}

Turning to the case $r_s(K) = 3$, recall from Lemmas~\ref{lem:c3} and~\ref{lem:r3} that there are two possible labeled knotoid diagrams $\K_{m,n}$ and $\K'_{m,n}$ corresponding to $K$.  These diagrams are shown in Figure~\ref{fig:rs3}.

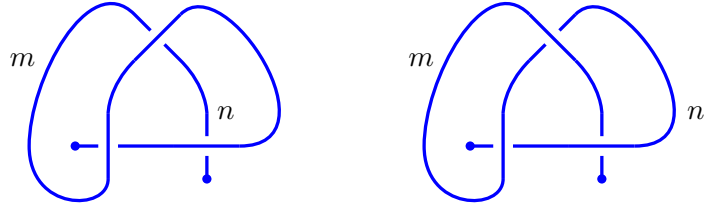
\begin{figure}[h!]
    \centering
    \includegraphics[width=0.6\textwidth]{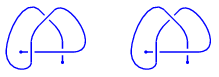}
                                        \put (-260,60) {$m$}
                                        \put (-182,40) {$n$}
                                        \put (-110,60) {$m$}
                                        \put (-5,40) {$n$}
     \caption{The only two labeled knotoid diagrams $\K_{m,n}$ (left) and $\K'_{m,n}$ (right) associated to knots $K$ with $r_s(K) = 3$.}
    \label{fig:rs3} 
\end{figure}

\begin{proposition}\label{prop:jones3}
Suppose $K$ is a knot with $r_s(K) = 3$ and $\det(K) = 25$.  There is a family $\{K_{m,n}\}$ of symmetric ribbon knots such that $K \in \{K_{m,n}\}$ and
\[ V_{K_{m,n}}(t) = (-t)^{m+n}(t^{-4}-2t^{-3}+3t^{-2}-4t^{-1}+4-4t+3t^2-2t^3+t^4)  + 1.\]
\end{proposition}

\begin{proof}
Suppose $r_s(K) = 3$ and $\det(K) = 25$.  Note that $\cl(\K_{m,n})$ is the unknot, and so by Lemma~\ref{lem:det} and the assumption that $\det(K) = 25$, we have that $K$ must correspond to the other family of diagrams, $\K'_{m,n}$ (the closure of which is the figure eight knot).  It follows that $K = K_{m,n}$, the knot arising from $\K'_{m,n}$, for some integers $m$ and $n$.  First, we note that performing an oriented smoothing in the twist region labeled $m$ yields $U_2$, as does performing an oriented smoothing in the twist region labeled $n$.  Thus, for any $m$, we can choose $L_+ = K_{m+2,n}$, $L_- = K_{m,n}$, and $L_0 = U_2$ to satisfy the oriented skein relation.  However, we can alternatively choose $L_+ = K_{m,n+2}$, $L_- = K_{m,n}$, and $L_0 = U_2$.  It follows that $V_{L_{m+2,n}}(t) = V_{L_{m,n+2}}(t)$.  A similar argument shows that $V_{L_{m-2,n}}(t) = V_{L_{m,n-2}}(t)$.  As a consequence, we need only show that the desired statement holds for knots of the form $K_{0,n}$ and $K_{1,n}$.

As above, we proceed by induction on $m$, verifying the formula for $K_{0,0}$, $K_{0,1}$, $K_{1,0}$, and $K_{1,1}$ by direct computation.  Suppose $n>1$ and suppose the formula holds for $K_{i,n-2}$, with $i=0$ or $1$.  Let $L_+ = K_{i,n}$, $L_- = K_{i,n-2}$, and $L_0 = U_2$, and let $f(t) = t^{-4}-2t^{-3}+3t^{-2}-4t^{-1}+4-4t+3t^2-2t^3+t^4$.  Applying the skein relation yields
\begin{eqnarray*}
V_{K_{i,n}}(t) &=& t^2 \, V_{K_{i,n-2}}(t) - (t^{1/2}-t^{3/2})V_{U_2}(t) \\
&=& t^2((-t)^{i+n-2} \cdot f(t)  + 1) - (t^{1/2}-t^{3/2})(-t^{-1/2}-t^{1/2}) \\
&=& (-t)^{i+n} \cdot f(t) + t^2 - t^2 +1 \\
&=& (-t)^{i+n} \cdot f(t) + 1.
\end{eqnarray*} 
Similarly, if $n < 0$, let $L_+ = K_{i,n}$, $L_- = K_{i,n+2}$, and $L_0 = U_2$.  Then
\begin{eqnarray*}
V_{K_{i,n}}(t) &=& t^{-2} \, V_{K_{i,n+2}}(t) + (t^{-3/2}-t^{-1/2})V_{U_2}(t) \\
&=& t^{-2}((-t)^{i + n+2} \cdot f(t)  + 1) + (t^{-3/2}-t^{-1/2})(-t^{-1/2}-t^{1/2}) \\
&=& (-t)^{i+n} \cdot f(t) + t^{-2} - t^{-2} +1 \\
&=& (-t)^{i+n} \cdot f(t) + 1.
\end{eqnarray*} 

\end{proof}

As a corollary, we show

\begin{corollary}\label{cor:1139}
The knots $K \in \{11n_{39}, 12n_{256}, 12n_{257},12n_{394},12n_{870}\}$ satisfy $r_s(K) \geq 4$.
\end{corollary}

\begin{proof}
From KnotInfo, $\det(K) = 25$, implying that $r_s(K) \geq 3$ by Corollary~\ref{cor:det}.  If $r_s(K) = 3$, then by Proposition~\ref{prop:jones3}, all coefficients $c_i$ of $V_K(t)$ satisfy $|c_i| \leq 5$, and at most one satisfies $|c_i| = 5$.  However, from KnotInfo,
\begin{eqnarray*}
V_{11n_{39}}(t) &=& -t^{-4}+ t^{-3}-t^{-1}+ 4-4t+ 5t^2-5t^3+ 4t^4-3t^5+ t^6 \\
V_{12n_{256}}(t) &=& -t^{-4}+ 2t^{-3}-2t^{-2}+ t^{-1} + 2-3t+ 5t^2-6t^3+ 6t^4-5t^5+ 3t^6-t^7 \\
V_{12n_{257}}(t) &=& t^{-6}-3t^{-5}+ 4t^{-4}-5t^{-3}+ 5t^{-2}-4t^{-1}+ 4-t+ t^3-t^4 \\
V_{12n_{394}}(t) &=& t^{-5}-t^{-4}+ t^{-2}-3t^{-1}+ 5-5t+ 5t^2-4t^3+ 3t^4-t^5\\
V_{12n_{870}}(t) &=& -t^{-3}-t^{-2}+ 2-3t+ 4t^2-5t^3+ 5t^4-4t^5+ 3t^6-t^7.
\end{eqnarray*}
Hence, $r_s(K) \geq 4$.
\end{proof}

\begin{remark}\label{rmk:search}
Proposition~\ref{prop:jones3} not only provides a new obstruction to $r_s(K) = 3$; it also helps us search for previously unknown symmetric ribbon disks.  Consider the symmetric ribbon knots $11n_{37}$ and $12n_{414}$, both of which have determinant 25.  Applying Prop~\ref{prop:upper} to the symmetric union presentations for these knots in~\cite{lamm2} yields $r_s(11n_{37}) \leq 4$ and $r_s(12n_{414}) \leq 4$.  Using KnotInfo, we have
\begin{eqnarray*}
V_{11n_{37}}(t) &=& t^{-6}-2t^{-5}+ 3t^{-4}-4t^{-3}+ 4t^{-2}-4t^{-1}+ 4-2t+ t^2 \\
V_{12n_{414}}(t) &=& -t^{-7}+ 2t^{-6}-3t^{-5}+ 4t^{-4}-4t^{-3}+ 4t^{-2}-3t^{-1}+ 3-t.
\end{eqnarray*}
Thus, by Proposition~\ref{prop:jones3}, if $r_s(11n_{37})=3$, then $11n_{37} = K_{m,n}$, where $m+n = -2$.  Similarly, if $r_s(12n_{414}) = 3$, then $12n_{414} = K_{m,n}$ with $m+n = -3$.  After searching through low-complexity possibilities, we verified with SnapPy that indeed $11n_{137} = K_{-3,1}$ and $12n_{414} = K_{-4,1}$, shown in Figure~\ref{fig:symm3}.  Note further that Lamm's symmetric union presentations for $11n_{37}$ and $12n_{414}$ are marked ``m" for minimal in Table 1 of~\cite{lamm2}.  This notion of minimality does not refer to symmetric ribbon number but rather to the crossing number of the symmetric union presentation.  Lamm's symmetric union presentation for $11n_{37}$ has symmetric ribbon number four with crossing number 11, while ours has symmetric ribbon number three and crossing number 12.  Similarly, Lamm's symmetric union presentation for $12n_{414}$ has symmetric ribbon number four with crossing number 12, while ours has symmetric ribbon number three and crossing number 13.  Finally, the partial knot for Lamm's presentations is $5_1$, while the partial knot for our presentations is $4_1$.
\end{remark}

\section{Tabulation of symmetric ribbon numbers}\label{sec:tab}

In this section, we put our tools to work, tabulating the symmetric ribbon numbers for symmetric ribbon knots up to 12 crossings.  As noted in the statement of Theorem~\ref{thm:main}, symmetric ribbon numbers for
\begin{itemize}
\item prime symmetric ribbon knots with 11 or fewer crossings appear in Table~\ref{table:10},
\item prime nonalternating symmetric ribbon knots with 12 crossings appear in Table~\ref{table:12n}, and
\item prime alternating symmetric ribbon knots with 12 crossings appear in Table~\ref{table:12a}.
\end{itemize}
In addition, lower bounds for the symmetric ribbon numbers of the 15 knots with 12 or fewer crossings that are not known to be symmetric are included in Table~\ref{table:non}.

For reference, each table row includes the knot determinant, Alexander polynomial, the lower bound for the ribbon number taken from~\cite{FMZ} (for knots with 11 or fewer crossings) or~\cite{polymath} (for knots with 12 crossings), the symmetric ribbon number (or a range of possible symmetric ribbon numbers), and the justification for the lower bound on the symmetric ribbon number.  Upper bounds for symmetric ribbon numbers are computed using Proposition~\ref{prop:upper} along with the symmetric union presentations from~\cite{lamm1} (for knots with 10 or fewer crossings) or~\cite{lamm2} (for knots with 11 or 12 crossings).  The only exceptions are the upper bounds for $10_{87}$, which comes from the diagram in~\cite{lamm3}, and the knots $11n_{37}$ and $12n_{414}$, noted in Remark~\ref{rmk:search} and Figure~\ref{fig:symm3}.

\begin{table}[ht]
\centering
	\resizebox*{!}{.9\dimexpr\textheight-1\baselineskip\relax}{%
\begin{tabular}{|l|l|l|l|l|l|}
\hline
$K$ & $\det(K)$ & $\Delta_K(t)$  & $r(K) \geq $ & $r_s(K)$ & lower \\ \hline
$6_1$ & 9 & $2-5t+2t^2$ & 2 & 2 & Lem.~\ref{lem:obv}  \\
$8_8$ & 25 & $2-6t+9t^2-6t^3+2$ & 3 & 3 & Lem.~\ref{lem:obv}  \\
$8_9$ & 25 & $1-3t+5t^2-7t^3+5t^4-3t^3+t^6$ & 3 & 3 & Lem.~\ref{lem:obv}  \\
$8_{20}$ & 9 & $1-2t+3t^2-2t^3+t^4$ & 2 & 2 & Lem.~\ref{lem:obv}  \\
$9_{27}$ & 49 & $1-5t+11t^2-15t^3+11t^4-5t^5+t^6$ & 3 & 4 & Lem.~\ref{lem:r3}  \\
$9_{41}$ & 49 & $3-12t+19t^2-12t^3+3t^4$ & 3 & 4 & Lem.~\ref{lem:r3}  \\
$9_{46}$ & 9 & $2-5t+2t^2$ & 2 & 2 & Lem.~\ref{lem:obv}  \\
$10_{3}$ & 25 & $6 -13t + 6t^2$ & 4 & 4 & Lem.~\ref{lem:obv}  \\
$10_{22}$ & 49 & $2-6t+10t^2-13t^3+10t^4-6t^5+2t^6$ & 4 & 4 & Lem.~\ref{lem:obv}  \\
$10_{35}$ & 49 & $2-12t+21t^2-12t^3+2t^4$ & 4 & 4 & Lem.~\ref{lem:obv}  \\
$10_{42}$ & 81 & $1-7t+19t^2-27t^3+19t^4-7t^5+t^6$ & 4 & 5 & Thm.~\ref{thm:r4} \\
$10_{48}$ & 49 & $1-3t+6t^2-9t^3+11t^4-9t^5+6t^6-3t^7+t^8$ & 4 & 4 & Lem.~\ref{lem:obv}  \\
$10_{75}$ & 81 & $1-7t+19t^2-27t^3+19t^4-7t^5+t^6$ & 4 & 5 & Thm.~\ref{thm:r4} \\
$10_{87}$ & 81 & $2-9t+18t^2-23t^3+18t^4-9t^5+2t^6$ & 4 & 5 & Thm.~\ref{thm:r4} \\
$10_{99}$ & 81 & $1-4t+10t^2-16t^3+19t^4-16t^5+10t^6-4t^7+t^8$ & 4 & 5 & Thm.~\ref{thm:r4} \\
$10_{123}$ & 121 & $1 - 6t + 15t^2 - 24t^3+29t^4-24t^5+15t^6-6t^7+t^8$ & 4 & 5  & Lem.~\ref{lem:obv}   \\
$10_{129}$ & 25 & $2-6t+9t^2-6t^3+2$ & 3 & 3 & Lem.~\ref{lem:obv}  \\
$10_{137}$ & 25 & $1 - 6t + 11t^2-6t^3+t^4$ & 3 & 3 & Lem.~\ref{lem:obv}  \\
$10_{140}$ & 9 & $1-2t+3t^2-2t^3+t^4$ & 2 & 2 & Lem.~\ref{lem:obv}  \\
$10_{153}$ & 1 & $1-t+t^2-3t^3+t^4-t^5+t^6$ & 3 & 3 & Lem.~\ref{lem:obv}  \\
$10_{155}$ & 25 & $1-3t+5t^2-7t^3+5t^4-3t^3+t^6$ & 3 & 3 & Lem.~\ref{lem:obv}  \\
$11a_{28}$ & 121 & $1 - 6t + 15t^2 - 24t^3+29t^4-24t^5+15t^6-6t^7+t^8$ & 4 & 5 & Thm.~\ref{thm:r4}  \\
$11a_{35}$ & 121 & $1-5t+14^2-25t^3+31t^4-25t^5+14t^6-5t^7+t^8$ & 4 & 5 & Thm.~\ref{thm:r4}  \\
$11a_{36}$ & 121 & $2-12t+28t^2-37t^3+28t^4-12t^5+t^6$ & 4 & 5 & Thm.~\ref{thm:r4}  \\
$11a_{58}$ & 81 & $2-9t+18t^2-23t^3+18t^4-9t^5+2t^6$ & 4 & 5, 6, 7, 8 & Thm.~\ref{thm:r4}  \\
$11a_{87}$ & 121. & $2-11t+28t^2-39t^3+28t^4-11t^5+2t^6$ & 4 & 5 & Thm.~\ref{thm:r4}  \\
$11a_{96}$ & 121 & $1-9t+29t^2-43t^3+29t^4-9t^5+t^6$ & 4 & 5 & Thm.~\ref{thm:r4}  \\
$11a_{115}$ & 121 & $3-13t+27t^2-35t^3 + 27t^4-13t^5+3t^6$ & 4 & 5 & Thm.~\ref{thm:r4}  \\
$11a_{164}$ & 169 & $1-7t+20t^2-35t^3+43t^4-35t^5+20t^6-7t^7+t^8$ & 4 & 5, 6 & Thm.~\ref{thm:r4}  \\
$11a_{169}$ & 121 & $2-12t+28t^2-37t^3+28t^4-12t^5+t^6$ & 4 & 5 & Thm.~\ref{thm:r4}  \\
$11a_{316}$ & 121 & $1-5t+14^2-25t^3+31t^4-25t^5+14t^6-5t^7+t^8$ & 4 & 5 & Thm.~\ref{thm:r4}  \\
$11a_{326}$ & 169 & $1-6t+19t^2-36t^3+45t^2-36t^5+19t^6-6t^7+t^8$ & 4 & 5, 6 & Thm.~\ref{thm:r4}  \\
$11n_{4}$ & 49 & $1-5t+11t^2-15t^3+11t^4-5t^5+t^6$ & 3 & 4 & Lem.~\ref{lem:r3}  \\
$11n_{21}$ & 49 & $1-5t+11t^2-15t^3+11t^4-5t^5+t^6$ & 3 & 4 & Lem.~\ref{lem:r3}  \\
$11n_{37}$ & 25 & $1-3t+5t^2-7t^3+5t^4-3t^3+t^6$ & 3 & 3 & Lem.~\ref{lem:obv} \\
$11n_{39}$ & 25 & $2-6t+9t^2-6t^3+2$ & 3 & 4 & Cor.~\ref{cor:1139}  \\
$11n_{42}$ & 1 & 1 & 3 & 3 & Lem.~\ref{lem:obv}  \\
$11n_{49}$ & 1 & $1-3t^2+t^4$ & 3 & 3 & Lem.~\ref{lem:obv} \\
$11n_{50}$ & 25 & $2-6t+9t^2-6t^3+2$ & 3 & 3 & Lem.~\ref{lem:obv} \\
$11n_{83}$ & 49 & $3-12t+19t^2-12t^3+3t^4$ & 3 & 4 & Lem~\ref{lem:r3}  \\
$11n_{116}$ & 1 & $1-3t^2+t^4$ & 3 & 3 & Lem~\ref{lem:r3}  \\
$11n_{132}$ & 25 & $2-6t+9t^2-6t^3+2$ & 3 & 3 & Lem~\ref{lem:r3}  \\
$11n_{139}$ & 9 & $2-5t+2t^2$ & 2 & 2 & Lem~\ref{lem:r3} \\
$11n_{172}$ & 49 & $1-5t+11t^2-15t^3+11t^4-5t^5+t^6$ & 3 & 4 & Lem~\ref{lem:r3} \\ \hline
\end{tabular}}
\caption{Symmetric ribbon number data and justifications for prime symmetric ribbon knots up to 11 crossings}
\label{table:10}
\end{table}

\begin{table}[!ht]
    \centering
	\resizebox*{!}{.93\dimexpr\textheight-1\baselineskip\relax}{%
    \begin{tabular}{|l|l|l|l|l|l|l|}
    \hline
        $K$ & $\text{det}(K)$ & $\Delta_K(t)$ & $r(K) \geq$ & $r_s(K)$ & lower \\ \hline
 $12n_4$ & 81 & $1-7t+19t^2-27t^3+19t^4-7t^5+t^6$ & 4 & 5 & Thm.~\ref{thm:r4} \\
 $12n_{19}$ & 1 & $1-3t-t^2+7t^3-t^4-3t^5+t^6$ & 4 & 4 & Lem.~\ref{lem:obv} \\
 $12n_{23}$ & 9 & $2-5t+2t^2$ & 3 & 4, 5 & Lem.~\ref{lem:r3}\\
 $12n_{24}$ & 49 & $1-5t+11t^2-15t^3+11t^4-5t^5+t^6$ & 3 & 4 & Lem.~\ref{lem:r3}\\
 $12n_{43}$ & 81 & $1-5t+10t^2-15t^3+19t^4-15t^5+10t^6-5t^7+t^8$ &  4 & 5 & Thm.~\ref{thm:r4} \\
 $12n_{48}$ & 49 & $2-12t+21t^2-12t^3+2t^4$  & 4 & 4 & Lem~\ref{lem:obv} \\
 $12n_{49}$ & 81 & $4-20t+33t^2-20t^3+4t^4$  & 4 & 5, 6, 7, 8 & Thm.~\ref{thm:r4} \\
  $12n_{87}$ & 49 & $4-12t+17t^2-12t^3+4t^4$ & 4 & 4 & Lem.~\ref{lem:obv} \\
 $12n_{106}$ & 81 & $1-4t+10t^2-16t^3+19t^4-16t^5+10t^6-4t^7+t^8$ & 4 & 5 & Thm.~\ref{thm:r4}\\
 $12n_{145}$ & 25 & $1-6t+11t^2-6t^3+t^4$ & 3 & 3 & Lem.~\ref{lem:obv} \\
 $12n_{170}$ & 81 & $6-20t+29t^2-20t^3+6t^4$  & 4 & 5 & Thm.~\ref{thm:r4} \\
 $12n_{214}$ & 1 & $1-t-t^2+3t^3-t^4-t^5+t^6$ & 3 & 3 & Lem.~\ref{lem:obv} \\
 $12n_{256}$ & 25 & $2-6t+9t^2-6t^3+2t^4$ & 3 & 4 & Cor.~\ref{cor:1139} \\
 $12n_{257}$ & 25 & $2-6t+9t^2-6t^3+2t^4$ & 3 & 4 & Cor.~\ref{cor:1139} \\
 $12n_{268}$ & 9 & $2-5t+2t^2$ &  3 & 4 & Lem.~\ref{lem:r3} \\
 $12n_{279}$ & 25 & $1-6t+11t^2-6t^3+t^4$  & 3 & 3 & Lem.~\ref{lem:obv} \\
 $12n_{288}$ & 49 & $4-12t+17t^2-12t^3+4t^4$ & 4 & 4 & Lem.~\ref{lem:obv} \\
 $12n_{309}$ & 1 & $1-t-t^2+3t^3-t^4-t^5+t^6$ & 3 & 3 & Lem.~\ref{lem:obv} \\
 $12n_{312}$ & 49 & $1-5t+11t^2-15t^3+11t^4-5t^5+t^6$  & 3 & 4 & Lem~\ref{lem:r3} \\
 $12n_{313}$ & 1 & $1$ & 3 & 3 & Lem.~\ref{lem:obv} \\
 $12n_{318}$ & 1 & $1-t-t^2+3t^3-t^4-t^5+t^6$ & 3 & 3 & Lem.~\ref{lem:obv} \\
 $12n_{360}$ & 49 & $3-12t+19t^2-12t^3+3t^4$  & 3 & 4 & Lem~\ref{lem:r3} \\
 $12n_{380}$ & 81 & $2-9t+18t^2-23t^3+18t^4-9t^5+2t^6$  & 4 & 5 & Thm.~\ref{thm:r4} \\
 $12n_{393}$ & 49 & $3-12t+19t^2-12t^3+3t^4$ & 3 & 4 & Lem.~\ref{lem:r3} \\
 $12n_{394}$ & 25 & $1-6t+11t^2-6t^3+t^4$  & 3 & 4 & Cor.~\ref{cor:1139} \\
 $12n_{397}$ & 49 & $1-5t+11t^2-15t^3+11t^4-5t^5+t^6$  & 3 & 4 & Lem.~\ref{lem:r3}\\
 $12n_{399}$ & 81 & $1-7t+19t^2-27t^3+19t^4-7t^5+t^6$ & 4 & 5 & Thm.~\ref{thm:r4} \\
 $12n_{414}$ & 25 & $2-6t+9t^2-6t^3+2t^4$  & 3 & 3 & Lem.~\ref{lem:obv} \\
 $12n_{420}$ & 81 & $1-7t+19t^2-27t^3+19t^4-7t^5+t^6$  & 4 & 5 & Thm.~\ref{thm:r4} \\
 $12n_{430}$ & 1 & $1$  & 3 & 3 & Lem.~\ref{lem:obv} \\
 $12n_{440}$ & 81 & $2-9t+18t^2-23t^3+18t^4-9t^5+2t^6$  & 4 & 5 & Thm.~\ref{thm:r4} \\
 $12n_{462}$ & 25 & $1-6t+11t^2-6t^3+t^4$ &  3 & 3 & Lem.~\ref{lem:obv} \\
 $12n_{501}$ & 49 & $4-12t+17t^2-12t^3+4t^4$  & 3 & 4 & Lem.~\ref{lem:r3} \\
 $12n_{504}$ & 121 & $1-6t+15t^2-24t^3+29t^4-24t^5+15t^6-6t^7+t^8$ & 4 & 5, 6 & Thm.~\ref{thm:r4} \\
 $12n_{553}$ & 81 & $4-20t+33t^2-20t^3+4t^4$ & 4 & 5 & Thm.~\ref{thm:r4} \\
 $12n_{556}$ & 81 & $4-20t+33t^2-20t^3+4t^4$ & 4 & 5 & Thm.~\ref{thm:r4} \\
 $12n_{582}$ & 9 & $1-2t+3t^2-2t^3+t^4$ & 2 & 2 & Lem.~\ref{lem:obv} \\
 $12n_{605}$ & 9 & $1-2t+4t^3-7t^4+4t^5-2t^7+t^8$ & 4 & 4 & Lem.~\ref{lem:obv} \\
 $12n_{636}$ & 81 & $1-7t+19t^2-27t^3+19t^4-7t^5+t^6$ & 4 & 5 & Thm.~\ref{thm:r4} \\
 $12n_{657}$ & 81 & $1-4t+9t^2-16t^3+21t^4-16t^5+9t^6-4t^7+t^8$ & 4 & 5 & Thm~\ref{thm:r4} \\
 $12n_{670}$ & 25 & $1-2t+3t^2-4t^3+5t^4-4t^5+3t^6-2t^7+t^8$ & 4 & 4 & Lem.~\ref{lem:obv} \\
 $12n_{676}$ & 9 & $2-2t-4t^2+9t^3-4t^4-2t^5+2t^6$ & 4 & 5 & Thm.~\ref{thm:r4} \\
 $12n_{702}$ & 121 & $2-12t+28t^2-37t^3+28t^4-12t^5+2t^6$ & 4 & 5 & Thm.~\ref{thm:r4} \\
 $12n_{706}$ & 49 & $1-4t+6t^2-8t^3+11t^4-8t^5+6t^6-4t^7+t^8$  & 4 & 4, 5, 6 & Lem.~\ref{lem:obv} \\
 $12n_{708}$ & 49 & $1-3t+6t^2-9t^3+11t^4-9t^5+6t^6-3t^7+t^8$ & 4 & 4 & Lem.~\ref{lem:obv}\\
 $12n_{721}$ & 25 & $1-2t+3t^2-4t^3+5t^4-4t^5+3t^6-2t^7+t^8$ & 4 & 4 & Lem.~\ref{lem:obv} \\
 $12n_{768}$ & 25 & $1-3t+5t^2-7t^3+5t^4-3t^5+t^6$ & 3 & 3 & Lem.~\ref{lem:obv} \\
 $12n_{782}$ & 81 & $2-8t+18t^2-25t^3+18t^4-8t^5+2t^6$ & 4 & 5 & Thm.~\ref{thm:r4} \\
 $12n_{802}$ & 121 & $1-5t+14t^2-25t^3+31t^4-25t^5+14t^6-5t^7+t^8$ & 4 & 5 & Thm.~\ref{thm:r4} \\
 $12n_{817}$ & 49 & $2-6t+10t^2-13t^3+10t^4-6t^5+2t^6$ & 4 & 4 & Lem.~\ref{lem:obv} \\
 $12n_{838}$ & 25 & $1-6t+11t^2-6t^3+t^4$ & 3 & 3 & Lem.~\ref{lem:obv} \\
 $12n_{870}$ & 25 & $1-3t+5t^2-7t^3+5t^4-3t^5+t^6$ & 3  & 4 & Cor.~\ref{cor:1139} \\
 $12n_{876}$ & 81 & $2-8t+18t^2-25t^3+18t^4-8t^5+2t^6$ &  4 & 5 & Thm.~\ref{thm:r4} \\ \hline
    \end{tabular}}
\caption{Symmetric ribbon number data and justifications for prime non-alternating symmetric ribbon knots with 12 crossings}
	\label{table:12n}
\end{table}

\begin{table}[!ht]
    \centering
	\resizebox*{!}{.92\dimexpr\textheight-1\baselineskip\relax}{%
\begin{tabular}{|l|l|l|l|l|l|l|}
\hline
$K$ & $\det(K)$ & $\Delta_K(t)$  & $r(K) \geq $ & $r_s(K)$ & lower  \\ \hline
 $12a_3$ & 169 & $2-14t+40t^2-57t^3+40t^4-14t^5+2t^6$ & 4 & 5, 6 & Thm.~\ref{thm:r4} \\
 $12a_{54}$ & 169 & $3-17t+39t^2-51t^3+39t^4-17t^5+3t^6)$ & 4 & 5 & Thm.~\ref{thm:r4} \\
 $12a_{77}$ & 225 & $1-7t+24t^2-49t^3+63t^4-49t^5+24t^6-7t^7+t^8$ & 4 & 6, 7, 8 & Thm.~\ref{thm:det} \\
 $12a_{100}$ & 225 & $3-21t+53t^2-71t^3+53t^4-21t^5+3t^6$ & 4 & 6, 7, 8 & Thm.~\ref{thm:det} \\
 $12a_{173}$ & 169 & $1-7t+20t^2-35t^3+43t^4-35t^5+20t^6-7t^7+t^8$ & 4 & 5, 6 & Thm.~\ref{thm:r4} \\
 $12a_{183}$ & 121 & $6-30t+49t^2-30t^3+6t^4$ & 5 & 5 & Lem.~\ref{lem:obv} \\
 $12a_{189}$ & 225 & $1-8t+26t^2-48t^3+59t^4-48t^5+26t^6-8t^7+t^8$ & 4 & 6, 7, 8 & Thm.~\ref{thm:det} \\
 $12a_{211}$ & 169 & $2-8t+20t^2-34t^3+41t^4-34t^5+20t^6-8t^7+2*t^8$ & 5 & 5 & Lem.~\ref{lem:obv} \\
 $12a_{221}$ & 169 & $2-14t+40t^2-57t^3+40t^4-14t^5+2t^6$ & 4 & 5 & Thm.~\ref{thm:r4} \\
 $12a_{245}$ & 225 & $1-7t+24t^2-49t^3+63t^4-49t^5+24t^6-7t^7+t^8$ & 4 & 6, 7, 8 & Thm.~\ref{thm:det} \\
 $12a_{258}$ & 169 & $1-7t+20t^2-35t^3+43t^4-35t^5+20t^6-7t^7+t^8$ & 4 & 5 & Thm.~\ref{thm:r4} \\
 $12a_{279}$ & 169 & $2-15t+40t^2-55t^3+40t^4-15t^5+2t^6$  & 5 & 5, 6 & Lem.~\ref{lem:obv} \\
 $12a_{377}$ & 225 & $1-8t+26t^2-48t^3+59t^4-48t^5+26t^6-8t^7+t^8$ & 4 & 6, 7, 8 & Thm.~\ref{thm:det} \\
 $12a_{425}$ & 81 & $6-20t+29t^2-20t^3+6t^4$ & 4 & 5 & Thm.~\ref{thm:r4} \\
 $12a_{427}$ & 225 & $1-8t+26t^2-48t^3+59t^4-48t^5+26t^6-8t^7+t^8$ & 4 & 6 & Thm.~\ref{thm:det} \\
 $12a_{435}$ & 225 & $1-8t+26t^2-48t^3+59t^4-48t^5+26t^6-8t^7+t^8$ & 4 & 6 & Thm.~\ref{thm:det}\\
 $12a_{447}$ & 121 & $2-12t+28t^2-37t^3+28t^4-12t^5+2t^6$ & 4 & 5 & Thm.~\ref{thm:r4} \\
 $12a_{456}$ & 225 & $1-8t+25t^2-48t^3+61t^4-48t^5+25t^6-8t^7+t^8$ & 5 & 6, 7, 8 & Thm.~\ref{thm:det}  \\
 $12a_{458}$ & 289 & $1-9t+32t^2-63t^3+79t^4-63t^5+32t^6-9t^7+t^8$ & 5 & 6, 7 & Thm.~\ref{thm:det} \\
 $12a_{464}$ & 225 & $1-7t+24t^2-49t^3+63t^4-49t^5+24t^6-7t^7+t^8$ & 4 & 6 & Thm.~\ref{thm:det} \\
 $12a_{473}$ & 289 & $1-8t+30t^2-64t^3+83t^4-64t^5+30t^6-8t^7+t^8$ & 5 & 6, 7 & Thm.~\ref{thm:det} \\
 $12a_{477}$ & 169 & $1-11t+41t^2-63t^3+41t^4-11t^5+t^6$ & 5 & 5 & Lem.~\ref{lem:obv}\\
 $12a_{484}$ & 289 & $1-9t+32t^2-63t^3+79t^4-63t^5+32t^6-9t^7+t^8$ & 5 & 6 & Thm.~\ref{thm:det} \\
  $12a_{606}$ & 169 & $4-18t+38t^2-49t^3+38t^4-18t^5+4t^6$ & 5 & 5 & Lem.~\ref{lem:obv} \\
 $12a_{631}$ & 225 & $4-22t+52t^2-69t^3+52t^4-22t^5+4t^6$ & 4 & 6 & Thm.~\ref{thm:det} \\
 $12a_{646}$ & 169 & $2-9t+21t^2-33t^3+39t^4-33t^5+21t^6-9t^7+2t^8$ & 5 & 5, 6 & Lem.~\ref{lem:obv} \\
 $12a_{667}$ & 121 & $2-7t+15t^2-23t^3+27t^4-23t^5+15t^6-7t^7+2t^8$ & 5 & 5 & Lem.~\ref{lem:obv} \\
 $12a_{715}$ & 169 & $4-18t+38t^2-49t^3+38t^4-18t^5+4t^6$ & 5 & 5, 6 & Lem.~\ref{lem:obv} \\
 $12a_{786}$ & 169 & $2-15t+40t^2-55t^3+40t^4-15t^5+2t^6$ & 5 & 5, 6 & Lem.~\ref{lem:obv} \\
 $12a_{819}$ & 169 & $1-5t+12t^2-21t^3+29t^4-33t^5+29t^6-21t^7+12t^8-5t^9+t^{10}$ & 5 & 5 & Lem.~\ref{lem:obv} \\
 $12a_{879}$ & 121 & $2-7t+15t^2-23t^3+27t^4-23t^5+15t^6-7t^7+2t^8$ & 5 & 5 & Lem.~\ref{lem:obv} \\
 $12a_{887}$ & 289 & $1-9t+32t^2-63t^3+79t^4-63t^5+32t^6-9t^7+t^8$ & 5 & 6, 7& Thm.~\ref{thm:det} \\
 $12a_{975}$ & 225 & $4-22t+52t^2-69t^3+52t^4-22t^5+4t^6$ & 4 & 6 & Thm.~\ref{thm:det} \\
 $12a_{979}$ & 225 & $2-10t+27t^2-46t^3+55t^4-46t^5+27t^6-10t^7+2t^8$ & 5 & 6, 7, 8 & Thm.~\ref{thm:det} \\
 $12a_{1011}$ & 121 & $1-4t+9t^2-15t^3+20t^4-23t^5+20t^6-15t^7+9t^8-4t^9+t^{10}$ & 5 & 5 & Lem.~\ref{lem:obv} \\
 $12a_{1019}$ & 361 & $1-10t+39t^2-80t^3+101t^4-80t^5+39t^6-10t^7+t^8$ & 5 & 6 & Thm.~\ref{thm:det} \\
 $12a_{1029}$ & 81 & $2-6t+10t^2-14t^3+17t^4-14t^5+10t^6-6t^7+2t^8$ & 5 & 5 & Lem.~\ref{lem:obv} \\
 $12a_{1034}$ & 121 & $8-30t+45t^2-30t^3+8t^4$ & 5 & 5 & Lem.~\ref{lem:obv}  \\
 $12a_{1083}$ & 169 & $2-9t+21t^2-33t^3+39t^4-33t^5+21t^6-9t^7+2t^8$ & 5 & 5, 6 & Lem.~\ref{lem:obv} \\
 $12a_{1087}$ & 225 & $1-8t+25t^2-48t^3+61t^4-48t^5+25t^6-8t^7+t^8$ & 5 & 6, 7, 8 & Thm.~\ref{thm:det} \\
 $12a_{1105}$ & 289 & $1-8t+30t^2-64t^3+83t^4-64t^5+30t^6-8t^7+t^8$ & 5 & 6 & Thm.~\ref{thm:det} \\
 $12a_{1119}$ & 169 & $2-9t+21t^2-33t^3+39t^4-33t^5+21t^6-9t^7+2t^8$ & 5 & 5 & Lem.~\ref{lem:obv} \\
 $12a_{1202}$ & 169 & $9-42t+67t^2-42t^3+9t^4$ & 5 & 5, 6 & Lem.~\ref{lem:obv} \\
 $12a_{1269}$ & 169 & $4-17t+38t^2-51t^3+38t^4-17t^5+4t^6$ & 5 & 5 & Lem.~\ref{lem:obv} \\
 $12a_{1277}$ & 121 & $4-14t+26t^2-33t^3+26t^4-14t^5+4t^6$ & 5 & 5 & Lem.~\ref{lem:obv} \\
 $12a_{1283}$ & 81 & $1-3t+6t^2-10t^3+13t^4-15t^5+13t^6-10t^7+6t^8-3t^9+t^{10}$ & 5 & 5 & Lem.~\ref{lem:obv} \\ \hline
\end{tabular}}
\caption{Symmetric ribbon number data and justifications for prime alternating symmetric ribbon knots with 12 crossings}
	\label{table:12a}
\end{table}

\begin{table}[!ht]
    \centering
    	\resizebox*{!}{.35\dimexpr\textheight-1\baselineskip\relax}
{%
    \begin{tabular}{|l|l|l|l|l|l|}
    \hline
            $K$ & $\text{det}(K)$ & $\Delta_K(t)$ & $r(K) \geq $ & $r_s(K) \geq$ & lower \\ \hline
 $11a_{103}$ & 81 & $4-20t+33t^2-20t^3-4t^4$ & 4 & 5 & Thm.~\ref{thm:r4} \\
  $11a_{165}$ & 81 & $2-9t+18t^2-23t^3+18t^4-9t^5+2t^6$ & 4 & 5 & Thm.~\ref{thm:r4} \\
 $11a_{201}$ & 81 & $4-20t+33t^2-20t^3-4t^4$ & 4 & 5 & Thm.~\ref{thm:r4} \\
 $11n_{67}$ & 9 & $2-5t+2t^2$ & 3 & 4 & Lem.~\ref{lem:r3} \\
 $11n_{73}$ & 9 & $1-2t+3t^2-2t^3+t^4$ & 3 & 4 & Lem.~\ref{lem:r3}\\
 $11n_{74}$ & 9 & $1-2t+3t^2-2t^3+t^4$ & 3 & 4 & Lem.~\ref{lem:r3} \\
 $11n_{97}$ & 9 & $2-5t+2t^2$ & 3 & 4 & Lem.~\ref{lem:r3} \\
 $12a_{348}$ & 225 & $2-17t+54t^2-79t^3+54t^4-17t^5+2t^6$ & 5 & 6 & Thm.~\ref{thm:det} \\
  $12a_{990}$ & 225 & $1-8t+26t^2-48t^3+59t^4-48t^5+26t^6-8t^7+t^8$ & 4 & 6 & Thm.~\ref{thm:det} \\ 
   $12a_{1225}$ & 225 & $1-5t+14t^2-28t^3+41t^4-47t^5+41t^6-28t^7+14t^8-5t^9+t^{10}$ & 5 & 6 & Thm.~\ref{thm:det} \\ 
$12n_{51}$ & 9 & $2-5t+2t^2$ & 3 & 4 & Lem.~\ref{lem:r3} \\  
 $12n_{56}$ & 9 & $1-2t+3t^2-2t^3+t^4$ & 3 & 4 & Lem.~\ref{lem:r3}   \\ 
 $12n_{57}$ & 9 & $1-2t+3t^2-2t^3+t^4$ & 3 & 4 & Lem.~\ref{lem:r3} \\ 
 $12n_{62}$ & 81 & $2-9t+18t^2-23t^3+18t^4-9t^5+2t^6$ & 4 & 5 & Thm.~\ref{thm:r4} \\
 $12n_{66}$ & 81 & $2-9t+18t^2-23t^3+18t^4-9t^5+2t^6$ & 4 & 5 & Thm.~\ref{thm:r4} \\ \hline
\end{tabular}}
	\caption{Prime ribbon knots that are not known to admit symmetric union presentations.  Note that the penultimate column is a lower bound for $r_s(K)$.}
	\label{table:non}
\end{table}

\section{Appendix: Proof of Theorem~\ref{thm:r4}}\label{sec:app}

We first note some overlap in the eight reduced singular arc diagrams $4_1,\dots,4_8$, which can be eliminated via leaf isotopies.  Figures~\ref{fig:simp1}, \ref{fig:simp5}, and \ref{fig:simp6} include singular arc diagrams with crossing information added at only one relevant double point, since the leaf isotopy can differ based on this crossing information.  In Figure~\ref{fig:simp1}, we see that two possible leaf isotopies convert a diagram of type $4_1^a$ to type $4_7$ or type $4_1^b$ to type $4_8$.  Similarly, in Figure~\ref{fig:simp5}, the two possible leaf isotopies convert a labeled knotoid diagram of type $4_5^a$ to type $4_4$ or type $4_5^b$ to type $4_7$.  Finally, in Figure~\ref{fig:simp6}, two possible leaf isotopies convert a labeled knotoid diagram of type $4_6^a$ to type $4_5$ or type $4_6^b$ to type $4_1$.  Therefore, after simplifying with leaf isotopies, we need only consider labeled knotoid diagrams of types $4_2$, $4_3$, $4_4$, $4_7$, and $4_8$.

\begin{figure}[h!]
    \centering
    \includegraphics[width=0.9\textwidth]{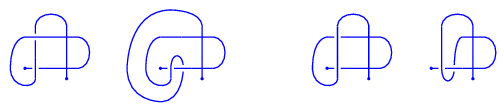}
                                \put (-355,-10) {$4_1^a$}
                                \put (-260,-10) {$4_7$}
                                \put (-122,-10) {$4_1^b$}
                                \put (-40,-10) {$4_8$}
                                \put (-315,45) {$\longrightarrow$}
                                \put (-80,45) {$\longrightarrow$}
     \caption{Two possible leaf isotopies on a labeled knotoid diagram of type $4_1$.}
    \label{fig:simp1}
\end{figure}

\begin{figure}[h!]
    \centering
    \includegraphics[width=0.9\textwidth]{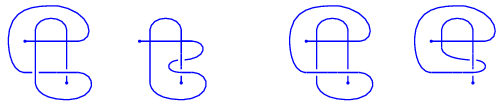}
                                \put (-355,-10) {$4_5^a$}
                                \put (-260,-10) {$4_4$}
                                \put (-132,-10) {$4_5^b$}
                                \put (-40,-10) {$4_7$}
                                \put (-311,45) {$\longrightarrow$}
                                \put (-95,45) {$\longrightarrow$}
     \caption{Two possible leaf isotopies on a labeled knotoid diagram of type $4_5$.  In the second frame, an arc of the diagram $4_4$ has been redrawn via planar isotopy in $S^2$.}
    \label{fig:simp5}
\end{figure}

\begin{figure}[h!]
    \centering
    \includegraphics[width=0.9\textwidth]{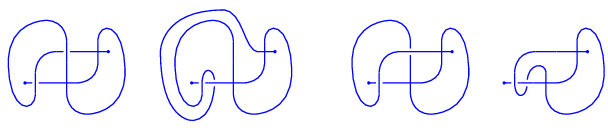}
                                \put (-351,-10) {$4_6^a$}
                                \put (-250,-10) {$4_5$}
                                \put (-132,-10) {$4_6^b$}
                                \put (-40,-10) {$4_1$}
                                \put (-308,45) {$\longrightarrow$}
                                \put (-85,45) {$\longrightarrow$}
                  \caption{Two possible leaf isotopies on a labeled knotoid diagram of type $4_6$.}
    \label{fig:simp6}
\end{figure}

\begin{proof}[Proof of Theorem~\ref{thm:r4}]
In order to find all possible Alexander polynomials of knots $K$ with $r_s(K) = 4$, it suffices to compute Alexander polynomials for all mod 2 labeled knotoid diagrams of types $4_2$, $4_3$, $4_4$, $4_7$, and $4_8$.  Each type involves two choices for each interior crossing and two choices for each of three strands, for a total of 32 possibilities.  Below, we list all distinct Alexander polynomials associated to each type.

\begin{eqnarray*}
4_2: && 1 \\
 && 1-3t^2+t^4 \\
 && 1-t-t^2+3t^3-t^4-t^5+t^6 \\
 && 1 -3t-t^2 +7t^3 -t^4 -3t^5 +t^6 \\
 && 1-3t+2t^2+3t^3-7t^4+3t^5+2t^6-3t^7+t^8
 \end{eqnarray*}

\begin{eqnarray*}
4_3: && 1 \\
 && 1-3t^2+t^4 \\
&& 1-t-t^2+3t^3-t^4-t^5+t^6 \\
  && 1-2t+3t^2 -4t^3 +5t^4 - 4t^5 + 3t^6 - 2t^7 + t^8 \\
  && 1 - 3t + 5t^2 - 7t^3 + 5t^4 - 3t^5 + t^6 \\
   && 1-4t+6t^2 -8t^3 +11t^4 -8t^5 +6t^6 -4t^7 +t^8 \\
   && 1-5t+11t^2 -15t^3 +11t^4 - 5t^6 + t^6 \\
     && 1 - 6t + 11t^2 - 6t^3 + t^4 \\
       && 2 - 6t + 9t^2 - 6t^3 + 2t^4 \\
  && 2-12t+21t^2 -12t^3 +2t^4 \\
 && 4-12t+17t^2 -12t^3 +4t^4 \\
  && 6-13t+6t^2
\end{eqnarray*}

\begin{eqnarray*}
4_4: && 1 \\
&& 1-2t^2 +3t^4 -2t^6 +t^8 \\
&& 1-t-t^3 +3t^4 -t^5 -t^7 +t^8 \\
&& 1-t+t^2 -3t^3 +t^4 -t^5 +t^6 \\
&& 1-2t+3t^2-2t^3+t^4  \\
&& 1-2t+3t^2 -4t^3 +5t^4 - 4t^5 + 3t^6 - 2t^7 + t^8 \\
&& 1 - 3t + 5t^2 - 7t^3 + 5t^4 - 3t^5 + t^6 \\
&& 1 - 6t + 11t^2 - 6t^3 + t^4 \\
&& 2-5t^2 +2t^4 \\
&& 2-5t+2t^2 \\
&& 2 - 6t + 9t^2 - 6t^3 + 2t^4 \\
&& 6-13t+6t^2
\end{eqnarray*}

\begin{eqnarray*}
4_7: && 1 \\
&& 1-2t^2 +3t^4 -2t^6 +t^8 \\
&& 1-t-t^3 +3t^4 -t^5 -t^7 + t^8 \\
&& 1-t-t^2+3t^3-t^4-t^5+t^6 \\
&& 1-t+t^2 -3t^3 +t^4 -t^5 +t^6 \\
&& 1-2t+t^2 +2t^3 -5t^4 +2t^5 +t^6 -2t^7 +t^8 \\
&& 1-2t+3t^2-2t^3+t^4  \\
&& 1 -3t-t^2 +7t^3 -t^4 -3t^5 +t^6 \\
&& 1-3t+6t^2 -9t^3 +11t^4 -9t^5 +6t^6 -3t^7 +t^8 \\
&& 1-5t+11t^2 -15t^3 +11t^4 - 5t^6 + t^6 \\
&& 2-5t^2 +2t^4 \\
&& 2 -3t-2t^2 +7t^3 -2t^4 -3t^5 +2t^6 \\
&& 2-5t+2t^2 \\
&& 2-6t+10t^2 -13t^3 +10t^4 -6t^5 +2t^6 \\
&& 2-12t+21t^2 -12t^3 +2t^4 \\
&& 3-12t+19t^2 -12t^3 +3t^4 \\
&& 4-12t+17t^2 -12t^3 +4t^4
\end{eqnarray*}

\begin{eqnarray*}
4_8: && 1 \\
&& 1-2t^2 +3t^4 -2t^6 +t^8 \\
&& 1-t-t^2+3t^3-t^4-t^5+t^6 \\
&& 1 - t - 3t^2 + 7t^3 - 3t^4 - t^5 + t^6 \\
&& 1-2t+4t^3 -7t^4 +4t^5 - 2t^7 + t^8 \\
&& 1-2t+t^2 +2t^3 -5t^4 +2t^5 +t^6 -2t^7 +t^8 \\
&& 1-2t+3t^2-2t^3+t^4  \\
&& 1 -3t-t^2 +7t^3 -t^4 -3t^5 +t^6 \\
&& 2-5t^2 +2t^4 \\
&& 2 -3t-2t^2 +7t^3 -2t^4 -3t^5 +2t^6 \\
&& 2-5t+2t^2
\end{eqnarray*}
These lists contain a total of 27 distinct Alexander polynomials, which are sorted by determinant in Table~\ref{table:r4}.
\end{proof}


\bibliographystyle{amsalpha}
\bibliography{symmetric.bib}

\end{document}